\documentclass[smallcondensed]{svjour3}
\usepackage{amsmath,amssymb,array,booktabs}
\usepackage{graphicx,graphics,epsfig}
\usepackage{grffile}
\usepackage{caption}
\usepackage{amsmath}
\usepackage{color}
\usepackage{verbatim}
\usepackage{amsfonts}
\usepackage{epsfig}
\usepackage{multirow}

\spnewtheorem{assumption}{Assumption}{\bf}{\it}
\spnewtheorem{corrollary}{Corrollary}{\bf}{\it}

\newcommand{\norm}[1]{\left\|#1\right\|}
\newcommand{\abs}[1]{\left|#1\right|}
\newcommand{\rset}{\mathbb{R}}

\journalname{Comput. Optim. Appl.}

\begin{document}

\title{A random coordinate descent algorithm for optimization problems
with composite objective function and linear coupled
constraints\thanks{The research leading to these results has
received funding from: the European Union (FP7/2007--2013) under
grant agreement no 248940; CNCS (project TE-231, 19/11.08.2010);
ANCS (project PN II, 80EU/2010); POSDRU/89/1.5/S/62557. \\
The authors thank Y. Nesterov and F. Glineur for inspiring
discussions. }}
\titlerunning{Random coordinate descent method for composite
linearly constrained optimization}

\author{Ion Necoara  \and Andrei Patrascu}
\authorrunning{I. Necoara  \and  A. Patrascu}

\institute{I. Necoara and A. Patrascu are  with the Automation and
Systems Engineering Department, University Politehnica Bucharest,
060042 Bucharest, Romania, Tel.: +40-21-4029195, \\ Fax:
+40-21-4029195; \email{\texttt{ion.necoara@acse.pub.ro,
andrei.patrascu@acse.pub.ro}}}

\date{Received:  25 March 2012 / Accepted: date}

\maketitle
\begin{abstract}
In this paper we propose a variant of the random coordinate descent
method for solving linearly constrained convex optimization problems
with composite objective functions. If the smooth part of the
objective function has Lipschitz continuous gradient, then we prove
that our method obtains  an $\epsilon$-optimal solution in ${\cal
O}(N^2/\epsilon)$ iterations, where $N$ is the number of blocks. For
the class of problems with cheap coordinate derivatives we show that
the new method is faster than methods based on full-gradient
information. Analysis for the rate of convergence in probability is
also provided. For strongly convex functions our method converges
linearly. Extensive numerical tests confirm that on very large
problems, our method is much more numerically efficient than methods
based on full gradient information.

\end{abstract}


\section{Introduction}
The basic problem of interest in this paper is the following convex
minimization problem with composite objective function:
\begin{equation} \label{model}
\begin{split}
 & \min\limits_{x \in \rset^n} F(x) \quad  \left(:= f(x) + h(x)\right) \\
 & \text{s.t.:} \;\;  a^T x=0,
\end{split}
\end{equation}
where $f:\rset^n \rightarrow \rset$ is a smooth convex function
defined by a black-box oracle,  $h:\rset^n \rightarrow \rset$ is a
general closed convex function  and $a \in \rset^{n}$. Further, we
assume that function $h$ is coordinatewise separable and simple (by
simple we mean that we can find a closed-form solution for the
minimization of  $h$ with some simple auxiliary function). Special
cases of this model include linearly constrained smooth optimization
(where $h \equiv 0$) which was  analyzed in
\cite{NecNes:12,XiaBoy:06}, support vector machines (where $h$ is
the indicator function of some box constraint  set)
\cite{HusKel:06,LisSim:05} and composite optimization (where $a=0$)
\cite{RicTak:11,TseYun:06,TseYun:07,TseYun:09}.

Linearly constrained optimization problems with composite objective
function arise in many applications such as  compressive sensing
\cite{CanRom:06}, image processing \cite{CheDon:01}, truss topology
design \cite{NesShp:12}, distributed control \cite{NecNed:11},
support vector machines \cite{TseYun:07}, traffic equilibrium and
network flow problems \cite{Ber:03} and many other areas. For
problems of moderate size there exist many iterative algorithms such
as Newton, quasi-Newton or projected gradient methods
\cite{DaiFle:06,FerMun:03,LinLuc:09}. However, the problems that we
consider in this paper have the following features: {\it the
dimension of the optimization variables} is very large such that
usual methods based on full gradient computations are prohibitive.
Moreover, {\it the incomplete structure of information} that may
appear when the data are distributed in space and time, or when
there exists  lack of physical memory and  enormous complexity of
the gradient update can also be an obstacle for full gradient
computations. In this case, it appears that a reasonable approach to
solving problem \eqref{model} is to use {\it (block) coordinate
descent methods}. These methods were among the first optimization
methods studied in literature \cite{Ber:99}. The main differences
between all variants of coordinate descent methods consist of the
criterion of choosing at each iteration the coordinate over which we
minimize our objective function and the complexity of this choice.
Two classical criteria, used often in these algorithms, are the
cyclic  and the greedy (e.g., Gauss-Southwell) coordinate search,
which significantly differ by the amount of computations required to
choose  the appropriate index. The rate of convergence  of cyclic
coordinate search methods has been determined recently in
\cite{BecTet:12,SahTew:12}. Also, for coordinate descent methods
based on the Gauss-Southwell rule, the convergence rate is given in
\cite{TseYun:06,TseYun:07,TseYun:09}. Another interesting approach
is based on random coordinate descent, where the coordinate search
is random. Recent complexity results on random coordinate descent
methods were obtained by Nesterov in \cite{Nes:10} for smooth convex
functions. The extension to composite objective functions  was given
in \cite{RicTak:11,RicTak:12} and for the grouped Lasso problem in
\cite{QinSch:10}. However, all these papers studied optimization
models where the constraint set is decoupled (i.e., characterized by
Cartesian product). The rate analysis of a random coordinate descent
method for linearly coupled constrained optimization problems with
smooth objective function was developed in \cite{NecNes:12}.

In this paper we present a  random coordinate descent method suited
for large scale problems with composite objective function.
Moreover, in our paper we focus on linearly coupled constrained
optimization problems (i.e., the constraint set  is coupled through
linear equalities). Note that the model considered in this paper is
more general than the one from \cite{NecNes:12}, since we allow
composite objective functions. We prove for our method an expected
convergence rate of order $\mathcal{O}(\frac{N^2}{k})$, where $N$ is
number of blocks and $k$ is the iteration counter. We show that for
functions with cheap coordinate derivatives the new method is much
faster, either in worst case complexity analysis, or numerical
implementation, than schemes based on full gradient information
(e.g., coordinate gradient descent method developed in
\cite{TseYun:09}). But our method also offers  other important
advantages, e.g., due to the randomization, our algorithm is easier
to analyze and implement, it leads to more robust output and is
adequate for modern computational architectures (e.g, parallel or
distributed architectures). Analysis for rate of convergence in
probability is also provided. For strongly convex functions we prove
that the new method converges linearly. We also provide extensive
numerical simulations and compare   our algorithm against
state-of-the-art methods from the literature on three large-scale
applications: support vector machine, the Chebyshev center of a set of
points and random generated optimization  problems with an
$\ell_1$-regularization term.

The paper is organized as follows.
In order to present our main results, we introduce some notations
and assumptions for problem \eqref{model} in Section
\ref{preliminaries}. In Section \ref{RCD} we present  the new random
coordinate descent  (RCD) algorithm. The main results of the paper
can be found in Section \ref{convresults}, where we derive the rate
of convergence in expectation, probability and for the strongly
convex case. In Section \ref{generalization} we generalize the
algorithm and extend the previous results to a more general model.
We also analyze its complexity and compare it  with other methods
from the literature, in particular  the coordinate descent method of
Tseng \cite{TseYun:09} in Section \ref{complanalysis}. Finally, we test the
practical efficiency of our algorithm through extensive numerical experiments
in Section \ref{numexp}.


\subsection{Preliminaries}\label{preliminaries}

We work in the space $\rset^n$ composed of column vectors.  For $x,y
\in \rset^n$ we denote:
\begin{equation*}
 \langle x,y \rangle = \sum\limits_{i=1}^n x_i y_i.
\end{equation*}
We use the same notation $\langle \cdot,\cdot \rangle$ for spaces of different dimensions.
If we fix a norm $\norm{\cdot}$ in $\rset^n$, then its dual norm is defined by:
\begin{equation*}
 \norm{y}^*= \max\limits_{\norm{x}=1} \langle y,x \rangle.
\end{equation*}
We assume that the entire space dimension is decomposable into $N$
blocks:
\begin{equation*}
 n=\sum\limits_{i=1}^N n_i.
\end{equation*}
We denote by $U_i$ the blocks of the identity matrix:
\begin{equation*}
 I_n = \left[ U_1 \ \dots \ U_N \right],
\end{equation*}
where $U_i \in \rset^{n \times n_i}$. For some vector $x \in
\rset^n$, we use the notation $x_i$ for the $i$th block in $x$,
$\nabla_i f(x)=U_i^T \nabla f(x)$ is the $i$th block in the
gradient of the function $f$ at $x$, and
$\nabla_{ij}f(x)=\begin{bmatrix} \nabla_i f(x) \\ \nabla_j f(x)
\end{bmatrix}$. We denote by $supp(x)$ the number of
nonzero coordinates in $x$. Given a matrix $A \in \rset^{m \times n}$,
we denote its nullspace by $Null(A)$.
In the rest of the paper  we consider
local Euclidean norms in all spaces $\rset^{n_i}$, i.e., $\norm{x_i}
=\sqrt{(x_i)^{T} x_i}$ for all $x_i \in \rset^{n_i}$ and $i=1,
\dots,N$.

\noindent For model \eqref{model} we make the following assumptions:
\begin{assumption}
\label{assumpbeg1}
The smooth and nonsmooth parts of the objective function in optimization model \eqref{model}  satisfy
the following properties:
\begin{enumerate}
\item[(i)] Function $f$ is convex and has block-coordinate Lipschitz continuous gradient:
\begin{equation*}
\norm{\nabla_i f(x+U_ih_i) - \nabla_if(x)} \le L_i \norm{h_i} \quad
\; \forall x \in \rset^n, \;\;  h_i \in \rset^{n_i}, \;\; i=1,
\dots, N.
\end{equation*}
\item[(ii)] The nonsmooth function $h$ is convex and coordinatewise separable.
\end{enumerate}
\end{assumption}
Assumption \ref{assumpbeg1} $(i)$ is typical for composite
optimization, see e.g., \cite{Nes:07,TseYun:09}. Assumption
\ref{assumpbeg1} $(ii)$ covers many applications as we further
exemplify. A special case of coordinatewise separable function that
has attracted a lot of attention in the area of signal processing
and data mining is the $\ell_1$-regularization \cite{QinSch:10}:
\begin{equation}\label{norm1}
  h(x) = \lambda \norm{x}_1,
\end{equation}
where $\lambda>0$.  Often, a large $\lambda$ factor induces sparsity
in the solution of optimization problem \eqref{model}. Note that the
function $h$ in \eqref{norm1} belongs to the general class of
coordinatewise separable piecewise linear/quadratic functions with
$\mathcal{O}(1)$ pieces. Another special case is the box indicator
function, i.e.:
\begin{equation}\label{indicator}
  h(x) = \mathbf{1}_{[l, u]} =
       \begin{cases}
          0, & l \le  x \le u\\
          \infty, &\text{otherwise}.
       \end{cases}
\end{equation}
Adding box constraints to a quadratic objective function $f$ in
\eqref{model} leads e.g., to  support vector machine (SVM) problems
\cite{ChaLin:11,TseYun:07}. The reader can easily find many other examples
of function $h$ satisfying Assumption \ref{assumpbeg1} $(ii)$.

\noindent Based on Assumption 1 $(i)$, the following inequality can
be derived \cite{Nes:04}:
\begin{equation}\label{assump1}
f(x+U_ih_i) \le f(x) + \langle\nabla_if(x), h_i \rangle +
\frac{L_i}{2}\norm{h_i}^2 \quad
  \forall x \in \rset^n, \;\;  h_i \in \rset^{n_i}.
\end{equation}
In the sequel, we use the notation:
\begin{equation*}
 L= \max\limits_{1\le i\le N}L_i.
\end{equation*}
For $\alpha \in [0,1]$ we introduce the extended norm on $\rset^n$
similar as   in \cite{Nes:10}:
\begin{equation*}
 \norm{x}_{\alpha} = \left(\sum\limits_{i=1}^N L_i^{\alpha}\norm{x_i}^2 \right)^{\frac{1}{2}}
\end{equation*}
and its dual norm
\begin{equation*}
 \norm{y}_{\alpha}^* = \left(\sum\limits_{i=1}^N \frac{1}{L_i^{\alpha}}
 \norm{y_i}^2 \right)^{\frac{1}{2}}.
\end{equation*}
Note that these norms satisfy the Cauchy-Schwartz inequality:
\begin{equation*}
 \norm{x}_{\alpha} \norm{y}_{\alpha}^* \ge \langle x, y \rangle \;\;  \forall x, y  \in \rset^n.
\end{equation*}

\noindent For a simpler exposition we use a context-dependent
notation as follows: let $x \in \rset^n$ such that $x = \sum_{i=1}^N
U_i x_i$, then $x_{ij} \in \rset^{n_i + n_j}$ denotes a two
component vector $x_{ij}=\begin{bmatrix} x_i
\\x_j \end{bmatrix}$. Moreover, by addition with a vector in the extended space $y\in
\rset^n$, i.e.,  $y+x_{ij}$, we understand $y+ U_ix_i+U_jx_j$. Also,
by inner product $\langle y, x_{ij}\rangle$ with vectors $y$ from
the extended space $\rset^n$ we understand $\langle y, x_{ij}\rangle
= \langle y_i, x_i\rangle  + \langle y_j, x_j\rangle$.

\noindent Based on Assumption \ref{assumpbeg1} $(ii)$ we can derive
from \eqref{assump1} the following result:
\begin{lemma} \label{lemmalij}
Let function $f$ be convex and satisfy Assumption \ref{assumpbeg1}.
Then, the function  $f$ has componentwise Lipschitz continuous
gradient w.r.t.  every pair $(i,j)$, i.e.:
\begin{equation*}
\norm{\begin{bmatrix}
\nabla_{i}f(x\!+\!U_is_i\!+\!U_js_j)\\
\nabla_{j}f(x\!+\!U_is_i\!+\!U_js_j)\\
\end{bmatrix}\!-\!
\begin{bmatrix}
\nabla_i f(x)\\
\nabla_j f(x)
\end{bmatrix}
}^*_{\alpha}\!\le\!
L_{ij}^{\alpha}\norm{
\begin{bmatrix}
s_i\\
s_j\\
\end{bmatrix}}_{\alpha}  \; \forall x \in \rset^n,\; s_i\in\rset^{n_i}, \; s_j\in\rset^{n_j},
\end{equation*}
where we define $L_{ij}^{\alpha} = L_i^{1-\alpha} + L_j^{1-\alpha}$.\\
\end{lemma}
\begin{proof}
Let $f^* = \min\limits_{x \in \rset^n} f(x)$.
Based on \eqref{assump1} we have for any pair $(i,j)$:
\begin{equation*}
\begin{split}
 f(x)-f^* &\ge \max\limits_{l \in \{1,\dots N\}} \frac{1}{2L_l}  \norm{\nabla_l f(x)}^{2}
 \ge \max\limits_{l \in \{i,j \} }\frac{1}{2L_l} \norm{\nabla_l f(x)}^{2}\\
& \ge \frac{1}{2 \left(L_i^{1-\alpha} + L_j^{1-\alpha} \right)}
\left(\!\frac{1}{L_i^{\alpha}}\norm{\nabla_if(x)}^{2}\!+\!\frac{1}{L_j^{\alpha}}
\norm{\nabla_jf(x)}^{2}\!\right)\\
& = \frac{1}{2L_{ij}^{\alpha}} \norm{\nabla_{ij}f(x)}^{*2}_{\alpha},
\end{split}
\end{equation*}
where in the third inequality we used that $\alpha a + (1-\alpha) b
\le \max\{a,b\}$ for all $\alpha \in [0,1]$. Now, note that the
function $g_1(y_{ij})=f(x+y_{ij}-x_{ij})-f(x)- \langle\nabla
f(x),y_{ij}-x_{ij}\rangle $ satisfies the Assumption
\ref{assumpbeg1} $(i)$. If we apply the above inequality to
$g_1(y_{ij})$ we get the following relation:
\begin{equation*}
f(x+y_{ij}-x_{ij}) \ge f(x) + \langle\nabla f(x),y_{ij}-x_{ij}\rangle +\frac{1}{2L_{ij}^{\alpha}}\norm{\nabla_{ij}f(x+y_{ij}-x_{ij}) - \nabla_{ij}f(x)}_{\alpha}^{*2}.
\end{equation*}
On the other hand, applying the same inequality to $g_2(x_{ij})=f(x)
- f(x+y_{ij}-x_{ij}) + \langle \nabla
f(x+y_{ij}-x_{ij}),y_{ij}-x_{ij}\rangle$, which also satisfies
Assumption \ref{assumpbeg1} $(i)$, we have:
\begin{align*}
f(x)\ge f(x+y_{ij}-x_{ij})+\langle\nabla f(x+&y_{ij}-x_{ij}),y_{ij}-x_{ij}\rangle+\\
& \frac{1}{2L_{ij}^{\alpha}}
\norm{\nabla_{ij}f(x+y_{ij}-x_{ij})-\nabla_{ij}f(x)}_{\alpha}^{*2}.
\end{align*}
Further, denoting $s_{ij}=y_{ij}-x_{ij}$ and adding up the resulting inequalities we
get:
\begin{equation*}
\begin{split}
 \frac{1}{L_{ij}^{\alpha}} \norm{ \nabla_{ij}f(x+s_{ij})-\nabla_{ij}f(x)}^{*2}_{\alpha} &\le\langle\nabla_{ij}f(x+s_{ij})
-\nabla_{ij} f(x), s_{ij}\rangle \\
& \le \norm{ \nabla_{ij} f(x+s_{ij}) - \nabla_{ij} f(x) }^*_{\alpha}
\norm{s_{ij}}_{\alpha},
\end{split}
\end{equation*}
for all $x \in \rset^n$ and $s_{ij} \in \rset^{n_i+n_j}$, which  proves the statement of this lemma. \qed
\end{proof}

\noindent It is straightforward to see that we can obtain from Lemma
\ref{lemmalij} the following inequality (see also \cite{Nes:04}):
\begin{equation}\label{lips2}
 f(x+s_{ij}) \le f(x) + \langle \nabla_{ij} f(x), s_{ij}\rangle + \frac{L_{ij}^{\alpha}}{2}\norm{s_{ij}}^2_ {\alpha},
\end{equation}
for all $x \in \rset^n, s_{ij} \in \rset^{n_i + n_j}$ and $\alpha
\in [0, \ 1]$.


\section{Random coordinate descent algorithm}\label{RCD}

In this section we introduce a variant of Random Coordinate Descent
(RCD) method for solving problem \eqref{model}  that performs a
minimization step with respect to two block variables at each
iteration. The coupling constraint (that is, the weighted sum
constraint $a^T x=0$) prevents the development of an algorithm that
performs a minimization with respect to only one variable at each
iteration. We will therefore be interested in the restriction of the
objective function $f$ on feasible directions consisting of at least
two nonzero (block) components.

Let $(i,j)$ be a two dimensional random variable, where $i, j \in
\{1, \dots, N \}$ with $ i \neq j$ and
$p_{i_kj_k}=\text{Pr}((i,j)=(i_k,j_k))$ be its probability
distribution. Given a feasible $x$, two blocks are chosen randomly
with respect to a probability distribution $p_{ij}$ and a quadratic
model derived from the composite objective function is minimized
with respect to these coordinates. Our method has the following
iteration: given a feasible initial point $x^0$, that is $a^Tx^0=0$,
then for all $k \geq 0$
 \begin{equation*}
\boxed{
\begin{split}
& \textbf{Algorithm 1 (RCD)}\\
&1.\ \text{Choose randomly two coordinates}\ (i_k,j_k) \ \text{with probability}\
 p_{i_kj_k}\\
 &2.\ \text{Set} \ x^{k+1} = x^k + U_{i_k} d_{i_k}+U_{j_k} d_{j_k},
 \end{split}}
 \end{equation*}
where the directions $d_{i_k}$ and $d_{j_k}$ are chosen as follows:
if we use for simplicity the notation $(i,j)$ instead of
$({i_k},{j_k})$, the direction $d_{ij} = [d_i^T \; d_j^T]^T$ is
given by
\begin{equation}\label{extdir}
\begin{split}
d_{ij}\!=\!\arg\!\!\min_{s_{ij}\in
\rset^{n_{i}\!+\!n_{j}}}\!&f(x^k)\!+\!\langle\nabla_{ij}f(x^k),s_{ij}\rangle\!+
\!\frac{L_{ij}^{\alpha}}{2}\norm{s_{ij}}^2_{\alpha}\!+\!
h(x^k+s_{ij})\\
\text{s.t.:} \;\;\  & a_i^Ts_i+a_j^Ts_j=0,\\
\end{split}
\end{equation}
where $a_i\in \rset^{n_i}$ and $a_j \in \rset^{n_j}$ are the $i$th
and $j$th blocks of vector $a$, respectively. Clearly, in our
algorithm we maintain feasibility at each iteration, i.e. $a^T x^k =
0$ for all $k \geq 0$.

\begin{remark}
\label{rem1} \hspace{1cm}
\begin{itemize}
\item[(i)] Note that for the scalar case (i.e., $N=n$) and $h$ given by
\eqref{norm1} or \eqref{indicator}, the direction $d_{ij}$ in
\eqref{extdir} can be computed in closed form. For the block case
(i.e., $n_i > 1$ for all $i$) and if $h$ is  a coordinatewise
separable, strictly convex and piece-wise linear/quadratic function
with $\mathcal{O}(1)$ pieces (e.g., $h$ given by \eqref{norm1}),
there are algorithms for solving the above subproblem in linear-time
(i.e., $\mathcal{O}(n_i+n_j)$ operations) \cite{TseYun:09}. Also for $h$
given by \eqref{indicator}, there exist in the literature algorithms
for solving the subproblem \eqref{extdir} with overall complexity
$\mathcal{O}(n_i+n_j)$  \cite{BerKov:93,Kiw:07}.

\item[(ii)] In algorithm (RCD) we consider $(i, j)=(j,  i)$ and $i \neq j$. Moreover,
we know that the complexity of choosing   randomly a pair $(i,j)$
with a uniform probability distribution requires $\mathcal{O}(1)$
operations. \qed
\end{itemize}
\end{remark}

\noindent We assume that random variables $(i_k, j_k)_{k \ge 0}$ are
i.i.d. In the sequel, we use notation $\eta^k$ for the entire
history of random pair choices and $\phi^k$ for the expected value
of the objective function w.r.t. $\eta^k$, i.e.:
$$    \eta^k = \left\lbrace (i_0, j_0), \dots, (i_{k-1},j_{k-1}) \right\rbrace \ \text{and} \
\phi^k=E_{\eta^k} \left[ F(x^k) \right].$$


\subsection{Previous work}\label{prevwork}

We briefly review some well-known methods  from the literature for
solving the optimization model \eqref{model}. In
\cite{TseYun:06,TseYun:07,TseYun:09} Tseng studied optimization problems in the
form \eqref{model} and developed a (block) coordinate gradient
descent(CGD) method based on the Gauss-Southwell  choice rule. The
main requirement for the (CGD) iteration is the solution of the
following problem: given a feasible $x$ and a working set of indexes
$\mathcal{J}$, the update direction is defined by
\begin{equation}\label{CGDiter}
 \begin{split}
 d_H(x;\mathcal{J})= \arg\min\limits_{s \in \rset^n} & f(x)+\langle\nabla f(x),s \rangle + \frac{1}{2}\langle Hs,s \rangle + h(x+s)\\
\text{s.t.:} \;\;  & a^Ts=0, \;\; s_j=0 \;\;\; \forall j \notin \mathcal{J},\\
 \end{split}
\end{equation}
where $H \in \rset^{n \times n}$ is a symmetric matrix chosen at the
initial step of the algorithm.
\begin{equation*}
\begin{split}
& \textbf{Algorithm (CGD):}\\
1.&\ \text{Choose a nonempty set of indices}\ \mathcal{J}^k \subset
\{1, \dots, n\} \
 \text{with respect to the} \\
 & \ \text{Gauss-Southwell rule}\\ 
2. &\ \text{Solve} \ \eqref{CGDiter}\ \text{with} \  x=x^k,\  \mathcal{J}=\mathcal{J}^k,\ H=H_k \
\text{to obtain} \ d^k=d_{H_k}(x^k;\mathcal{J}^k) \\
3.&\ \text{Choose stepsize } \alpha^k>0 \ \text{and set} \
x^{k+1}=x^k+ \alpha^k d^k.
\end{split}
\end{equation*}
In \cite{TseYun:09}, the authors proved for the particular case when
function $h$ is piece-wise linear/quadratic with $\mathcal{O}(1)$
pieces that an $\epsilon$-optimal solution is attained in
$\mathcal{O}(\frac{nLR_0^2}{\epsilon})$ iterations, where $R_0$  denotes the
Euclidean  distance from the initial point to some optimal solution.
Also, in \cite{TseYun:09}  the authors derive estimates  of order
$\mathcal{O}(n)$ on the computational  complexity of each iteration
for this choice of $h$.

Furthermore, for a quadratic function $f$  and a  box indicator
function $h$ (e.g., support vector machine (SVM) applications) one
of the first decomposition approaches developed similar to (RCD) is
Sequential Minimal Optimization (SMO) \cite{Pla:99}.  SMO consists
of choosing at each iteration two scalar coordinates with respect to
some heuristic rule based on KKT conditions and solving the small QP
subproblem obtained through the decomposition process. However, the
rate of convergence is not provided for the SMO algorithm. But the
numerical experiments show that the method is very efficient in
practice due to the closed form solution of the QP subproblem. List
and Simon \cite{LisSim:05} proposed a variant of block coordinate
descent method for which  an arithmetic  complexity of order
$\mathcal{O}(\frac{n^2 L R_0^2}{\epsilon})$ is proved on a quadratic
model with a box indicator function and generalized linear
constraints. Later, Hush et al. \cite{HusKel:06} presented a more
practical decomposition method which attains the same complexity as
the previous methods.

A random coordinate descent algorithm for model \eqref{model} with
$a=0$ and $h$ being the indicator function for a Cartesian product
of sets was analyzed by Nesterov in \cite{Nes:10}. The generalization
of this algorithm to composite objective functions has been studied
in \cite{QinSch:10,RicTak:11}. However, none of these papers studied the
application of coordinate descent algorithms to linearly coupled
constrained optimization models. A similar random coordinate descent
algorithm as the (RCD) method described in the present paper, for
optimization problems with smooth objective and linearly coupled
constraints, has been developed and analyzed by Necoara et al. in
\cite{NecNes:12}. We further extend these results to linearly
constrained composite objective function optimization and provide in
the sequel the convergence rate analysis for the previously
presented variant of the (RCD) method (see Algorithm 1 (RCD)).


\section{Convergence results}\label{convresults}

In the following subsections we derive the convergence rate of
Algorithm 1 (RCD) for composite optimization model \eqref{model} in
expectation, probability and for strongly convex functions.


\subsection{Convergence in expectation}\label{convexp}
In this section we study the rate of convergence in expectation
of algorithm (RCD).
We consider uniform probability distribution, i.e.,
the event of choosing a pair $(i,j)$ can occur with probability:
\[ p_{ij}=\frac{2}{N(N-1)}, \]
since we assume that $(i,j)=(j,i)$ and $i \neq j \in \{1, \dots,N\}$
(see Remark \ref{rem1} (ii)). In order to provide the convergence
rate of our algorithm, first we have to define the {\it conformal
realization} of a vector introduced in \cite{Roc:67,Roc:84}.
\begin{definition}
Let $d,d'\in \rset^n$, then the vector $d'$ is {\it conformal} to $d$ if:
\begin{equation*}
  supp(d') \subseteq supp(d) \quad \text{and} \quad d'_j d_j \ge 0 \;\;  \forall j = 1, \dots, n.
\end{equation*}
\end{definition}
For a given matrix $A \in \rset^{m \times n}$, with $m \leq n$, we
introduce the notion of elementary vectors defined as:
\begin{definition}
 An elementary vector of $Null(A)$ is a vector $d \in Null(A)$
for which there is no nonzero vector $d' \in Null(A)$
conformal to $d$ and $supp(d')\neq supp(d)$.
\end{definition}
Based on Exercise 10.6 in \cite{Roc:84} we state the following lemma:
\begin{lemma}\label{lemmaconf}\cite{Roc:84}
Given $d\in Null(A)$, if $d$ is an elementary vector, then
$\abs{supp(d)}\le rank(A)+1 \le m+1$. Otherwise,
$d$ has a {\it conformal realization}:
\begin{equation*}
d= d^1 + \dots + d^s,
\end{equation*}
where $s\ge 1$ and $d^t  \in \text{Null}(A)$ are elementary vectors
conformal to $d$ for all $t=1,\dots, s$.
\end{lemma}
For the scalar case, i.e., $N=n$ and $m=1$, the method provided in
\cite{TseYun:09} finds a conformal realization with dimension $s \le
|\text{supp}(d)|-1$ within $\mathcal{O}(n)$ operations. We observe
that elementary vectors $d^t$ in Lemma \ref{lemmaconf} for the case
$m=1$ (i.e., $A=a^T$) have at most $2$ nonzero components.

\noindent Our convergence analysis is based on the following lemma,
whose proof can be found in \cite[Lemma 6.1]{TseYun:09}:
\begin{lemma}\label{lemma6.1}\cite{TseYun:09}
Let $h$ be coordinatewise separable and convex. For any $x, x+d \in
\text{dom}\  h$, let $d$ be expressed as $d=d^1+\dots+d^s$  for some
$s \ge 1$ and some nonzero $d^t \in \rset^n$ conformal to $d$ for
$t=1, \dots, s$. Then, we have:
\begin{equation*}
 h(x+d) - h(x) \ge \sum\limits_{t=1}^s \left( h(x+d^t)- h(x)\right).
\end{equation*}
\end{lemma}

\noindent For the simplicity of the analysis we introduce the
following linear subspaces: \[ S_{ij}=\left\lbrace d\in
\rset^{n_i+n_j}: \;\ a_{ij}^Td=0 \right\rbrace \quad  \text{and}
\quad S=\left\lbrace d\in \rset^{n}: \;\  a^Td=0 \right\rbrace.\]

\noindent A simplified update rule of algorithm (RCD) is expressed
as: \[ x^+=x+U_id_i+U_jd_j. \] We denote by $F^*$ and $X^*$ the
optimal value and the optimal solution set for problem
\eqref{model}, respectively. We also introduce the maximal residual
defined in terms of the norm $\| \cdot\|_\alpha$:
\begin{equation*}
R_{\alpha}=\max\limits_x  \left\{ \max\limits_{x^* \in X^*} \norm{x
- x^*}_{\alpha}: \;\; F(x)\le F(x^0) \right\},
\end{equation*}
which measures the size of the level set of $F$ given by $x^0$. We
assume that this distance is finite for the initial iterate $x^0$.

\noindent Now, we prove the main result of this section:
\begin{theorem}
\label{convergence1}
 Let $F$ satisfy Assumption \ref{assumpbeg1}.
Then, the random coordinate descent algorithm (RCD) based on the
uniform distribution generates a sequence $x^k$ satisfying the
following convergence rate for the expected values of the objective
function:
\begin{equation*}
   \phi^k - F^* \le \frac{N^2 L^{1-\alpha}R_{\alpha}^2}{k+\frac{N^2 L^{1-\alpha}
   R_{\alpha}^2}{F(x^0)-F^*}}.
\end{equation*}
\end{theorem}

\begin{proof}
For simplicity, we drop the index $k$ and use instead of $(i_k,
j_k)$ and $x^k$ the notation $(i,j)$ and $x$, respectively. Based on
\eqref{lips2} we derive:
\begin{equation*}
 \begin{split}
  F(x^+)& \le f(x) + \langle \nabla_{ij} f(x), d_{ij}\rangle + \frac{L_{ij}^{\alpha}}{2}\norm{d_{ij}}^2_{\alpha} + h(x+d_{ij})\\
  & \overset{\eqref{extdir}}{=} \min_{s_{ij}\in S_{ij}} f(x) + \langle \nabla_{ij} f(x), s_{ij}\rangle + \frac{L_{ij}^{\alpha}}{2}\norm{s_{ij}}^2_{\alpha} + h(x+s_{ij}).
\end{split}
\end{equation*}
Taking expectation in both sides w.r.t. random variable $(i,j)$ and recalling that $p_{ij}=\frac{2}{N(N-1)}$, we get:
\begin{equation} \label{sijineq}
\begin{split}
E_{ij}\left[ F(x^+) \right]&\le E_{ij}\left[\! \min_{s_{ij}\in S_{ij}} \! f(x)+\!\langle \nabla_{ij}f(x),s_{ij} \rangle+
\frac{L_{ij}^{\alpha}}{2} \norm{s_{ij}}^2_{\alpha}+ h(x+s_{ij}) \right]\\
&\le E_{ij}\left[ f(x)+ \langle \nabla_{ij}f(x),s_{ij} \rangle+
\frac{L_{ij}^{\alpha}}{2} \norm{s_{ij}}^2_{\alpha}+ h(x+s_{ij}) \right]\\
&=\!\frac{2}{N(N\!-\!1)}\!\sum\limits_{i,j}\!\left(\! f(x)\!+\!\langle
\nabla_{ij} f(x),s_{ij} \rangle\!+\!
\frac{L_{ij}^{\alpha}}{2}\! \norm{s_{ij}}_{\alpha}^2 \!+\! h(x+s_{ij})\! \right)\\
&=\!f(x)\!+\! \frac{2}{N\!(N\!\!-\!\!1)}\!\left(\!\langle\nabla
f(x),\!\sum_{i,j}\!
s_{ij}\rangle\!+\!\sum_{i,j}\!\frac{L_{ij}^{\alpha}}{2}\!\norm{s_{ij}}^2_{\alpha}
\!+\!\sum_{i,j}\!h(x\!\!+\!\!s_{ij})\!\right)\!,
\end{split}
\end{equation}
for all possible $s_{ij} \in S_{ij}$ and pairs $(i,j)$ with $ i \neq
j \in \{1, \dots, N\}$.

\noindent Based on Lemma \ref{lemmaconf} for $m=1$, it follows that
any $d \in S$ has a conformal realization defined by
$d=\sum\limits_{t=1}^{s} d^t$, where the vectors $d^t \in S$ are
conformal to $d$ and have only two nonzero components. Thus, for any
$t=1, \dots, s$ there is a pair $(i,j)$ such that $d^t \in S_{ij}$.
Therefore, for any $d \in S$ we can choose an appropriate set of
pairs $(i,j)$ and vectors $s_{ij}^d \in S_{ij}$ conformal to $d$
such that $d=\sum\limits_{i,j} s_{ij}^d$. As we have seen, the above
chain of relations in \eqref{sijineq} holds for any set of pairs
$(i,j)$ and vectors $s_{ij} \in S_{ij}$.  Therefore, it also holds
for the set of pairs $(i,j)$ and vectors $s_{ij}^d$ such that $d=
\sum\limits_{i,j}s_{ij}^d$. In conclusion, we have from
\eqref{sijineq} that:
\begin{equation*}
 E_{ij} \left[ F(x^+) \right]\le f(x) +\frac{2}{N(N-1)}\left(\!\langle \nabla f(x), \sum_{i,j}
s_{ij}^d\rangle\!+\!\sum_{i,j}\frac{L_{ij}^{\alpha}}{2}\norm{s_{ij}^d}^2_{\alpha}
+\sum_{i,j}\!h(x+s_{ij}^d)\right)\!,
\end{equation*}
for all $d \in S$. Moreover, observing that $L_{ij}^{\alpha} \le
2L^{1-\alpha}$ and   applying  Lemma \ref{lemma6.1} in the previous
inequality for coordinatewise separable functions
$\norm{\cdot}^2_{\alpha}$ and $h(\cdot)$,  we obtain:
\begin{equation}\label{eqconv}
\begin{split}
\!E_{ij}\!\left[\!F(x^+)\!\right]\!\le\!f(x&)\!+\!\frac{2}{N(N\!\!-\!\!1)}\!(\langle\nabla f(x),\sum_{i,j}
s_{ij}^d\rangle\!+\!\sum_{i,j}\frac{L_{ij}^{\alpha}}{2} \lVert s_{ij}^d\rVert^2_{\alpha}
\!\!+\!\!\sum_{i,j}\!h(x\!+\!s_{ij}^d))\\
\overset{\text{Lemma} \ 3}{\le}& f(x) + \frac{2}{N(N-1)} (\langle \nabla f(x),
 \sum_{i,j}s_{ij}^d\rangle + L^{1-\alpha} \lVert\sum_{i,j}s_{ij}^d\rVert^2_{\alpha} + \\
& \qquad \qquad \qquad \qquad \ \ \ \ h(x+\sum_{i,j}s_{ij}^d)\!+\!(\frac{N(N-1)}{2}-1)h(x))\\
 \overset{d= \sum\limits_{i,j} s_{ij}^d}{=}& (1-\frac{2}{N(N-1)})F(x)+ \frac{2}{N(N-1)} (f(x)+\langle \nabla f(x),d\rangle+\\
& \qquad \qquad \qquad\qquad \qquad \qquad \qquad \qquad
L^{1-\alpha}\norm{d}^2_{\alpha}\!+\!h(x+d)),
\end{split}
\end{equation}
for any $d \in S$. Note that \eqref{eqconv} holds for every $d \in
S$ since \eqref{sijineq} holds for any $s_{ij} \in S_{ij}$.
Therefore, as \eqref{eqconv} holds for every vector from the
subspace $S$, it also  holds  for the following particular vector
$\tilde{d} \in S$ defined as:
\begin{equation*}
 \tilde{d}= \arg\min_{s \in S} f(x)+\langle \nabla f(x),s \rangle +L^{1-\alpha}
\norm{s}^2_{\alpha} +h(x+s).
\end{equation*}
Based on this choice and using similar reasoning as in
\cite{Nes:07,RicTak:11} for proving the convergence rate of gradient type
methods for composite objective functions, we derive the following:
\begin{align*}
& f(x)+ \langle \nabla f(x), \tilde{d}\rangle + L^{1-\alpha}\norm{\tilde{d}}^2_{\alpha}+h(x+\tilde{d})\\
& =\min_{y \in S} f(x)+\langle \nabla f(x),y-x \rangle + L^{1-\alpha}\norm{y-x}^2_{\alpha}+h(y)\\
& \le \min_{y \in S} F(y) + L^{1-\alpha} \norm{y-x}^2_{\alpha} \\
& \le \min_{\beta \in [0,1]} F(\beta x^* + (1-\beta)x) + \beta^2
L^{1-\alpha} \norm{x-x^*}^2_{\alpha}\\
& \le \min_{\beta \in [0,1]} F(x) -\beta(F(x)-F^*) + \beta^2 L^{1-\alpha} R_{\alpha}^2\\
& = F(x) - \frac{(F(x)-F^*)^2}{ L^{1-\alpha}R_{\alpha}^2},
\end{align*}
where in the first inequality we used the convexity of $f$ while in
the second and third inequalities we used basic optimization
arguments. Therefore, at each iteration $k$ the following inequality
holds:
\begin{equation*}
\begin{split}
 E_{i_kj_k}\left[ F(x^{k+1}) \right] \le & (1-\frac{2}{N(N-1)})F(x^k)+\\
& \;\; \frac{2}{N(N-1)}\left[F(x^k) - \frac{(F(x^k) -F^*)^2}{L^{1-\alpha}R_{\alpha}^2}\right].\\
\end{split}
\end{equation*}
Taking expectation with respect to $\eta_{k}$ and using convexity properties we get:
\begin{equation}\label{probrel1}
\begin{split}
\phi^{k+1} - F^* \le &(1-\frac{2}{N(N-1)})(\phi^k-F^*)+\\
& \quad \frac{2}{N(N-1)}\left[(\phi^k-F^*) - \frac{(\phi^k -F^*)^2}{L^{1-\alpha}R_{\alpha}^2}\right]\\
\le & (\phi^k-F^*)-\frac{2}{N(N-1)}\left[\frac{(\phi^k -F^*)^2}{L^{1-\alpha}R_{\alpha}^2}\right].\\
\end{split}
\end{equation}
Further, if we denote $\Delta^k=\phi^k-F^*$ and $\gamma=N(N-1)L^{1-\alpha}R^2_{\alpha}$ we get:
\begin{equation*}
 \Delta^{k+1} \le \Delta^k - \frac{(\Delta^k)^2}{\gamma}.
\end{equation*}
Dividing both sides with $\Delta^k\Delta^{k+1} > 0$ and using the fact that $\Delta^{k+1} \le \Delta^k$ we get:
\begin{equation*}
\frac{1}{\Delta^{k+1}} \ge \frac{1}{\Delta^k} + \frac{1}{\gamma}
\;\;\;  \forall k\ge 0.
\end{equation*}
Finally, summing up from $0, \dots, k$ we easily get the above
convergence rate. \qed
\end{proof}

Let us analyze the convergence rate of our method for the two most
common cases of the extended norm introduced in this section:  w.r.t. extended Euclidean norm
$\norm{\cdot}_0$ (i.e., $\alpha=0$) and norm  $\norm{\cdot}_1$ (i.e.,
$\alpha=1$). Recall that the norm $\norm{\cdot}_1$ is defined by:
\begin{equation*}
 \norm{x}_1^2=\sum\limits_{i=1}^N L_i \norm{x_i}^2.
\end{equation*}

\begin{corrollary}
Under the same assumptions of Theorem \ref{convergence1}, the
algorithm (RCD) generates  a sequence $x^k$ such that the expected
values of the objective function satisfy the following convergence
rates for $\alpha=0$ and $\alpha = 1$:
\begin{equation*}
\begin{split}
&{ \alpha=0:} \hspace{20pt} \phi^k - F^* \le \frac{N^2 L R_0^2}{k+ \frac{N^2 L R_0^2}{F(x^0) - F^*}},\\
&{ \alpha=1:} \hspace{20pt} \phi^k - F^* \le \frac{N^2 R_1^2}{k+ \frac{N^2 R_1^2}{F(x^0) - F^*}}.
\end{split}
\end{equation*}
\end{corrollary}

\begin{remark}
We usually have  $R_1^2 \le L R_0^2$ and this shows the advantages
that the general  norm $\norm{\cdot}_{\alpha}$ has over the
Euclidean norm. Indeed,  if  we denote by $r_i^2 =\max_{x}
\{\max_{x^* \in X^*} \norm{x_i - x^*_i}^2: \; F(x) \le F(x^0) \}$,
then we can provide  upper bounds on $R_1^2 \leq \sum_{i=1}^N L_i
r_i^2$ and $R_0^2 \leq \sum_{i=1}^N r_i^2$. Clearly, the following
inequality is valid:
\begin{equation*}
 \sum\limits_{i=1}^N L_ir_i^2 \le \sum\limits_{i=1}^N L r_i^2,
\end{equation*}
 and the inequality holds with  equality only for $L_i=L$  for all $i= 1, \dots, N$. We recall that $L=
\max_i L_i$.  Therefore, in the majority of cases the estimate for
the rate of convergence based on  norm $\norm{\cdot}_1$ is much
better than that based on the Euclidean  norm $\norm{\cdot}_0$.
\end{remark}


\subsection{Convergence for strongly convex functions}

Now, we assume that the objective function in \eqref{model} is $\sigma_{\alpha}$-strongly convex with
respect to norm $\norm{\cdot}_{\alpha}$, i.e.:

\begin{equation}\label{strong}
F(x) \ge F(y) + \langle F'(y), x-y \rangle +
\frac{\sigma_{\alpha}}{2}\norm{x-y}_{\alpha}^2 \quad \forall x, y \in \rset^n,
\end{equation}
where $F'(y)$ denotes some subgradient of $F$ at $y$. Note that if
the function $f$ is $\sigma$-strongly convex w.r.t. extended
Euclidean norm, then  we can remark that it is also
$\sigma_{\alpha}$-strongly convex function w.r.t. norm
$\norm{\cdot}_{\alpha}$ and the following relation between the
strong convexity constants holds:
\begin{equation*}
\begin{split}
\frac{\sigma}{L^{\alpha}} \sum\limits_{i=1}^N L^{\alpha}\norm{x_i-y_i}^2  &\ge
\frac{\sigma}{L^{\alpha}} \sum\limits_{i=1}^N L_i^{\alpha} \norm{x_i-y_i}^2\\
&\ge \sigma_{\alpha} \norm{x-y}^2_{\alpha},
\end{split}
\end{equation*}
which leads to
\begin{equation*}
 \sigma_{\alpha} \le \frac{\sigma}{L^{\alpha}}.
\end{equation*}

\noindent Taking $y=x^*$ in \eqref{strong} and from optimality
conditions $\langle F'(x^*), x-x^* \rangle \ge 0$ for all $x \in S$
we obtain:
\begin{equation}\label{strong2}
 F(x) - F^* \ge \frac{\sigma_{\alpha}}{2}\norm{x-x^*}_{\alpha}^2.
\end{equation}

\noindent Next, we state the convergence result of our algorithm
(RCD) for solving the problem \eqref{model} with
$\sigma_{\alpha}$-strongly convex objective w.r.t.  norm
$\norm{\cdot}_{\alpha}$.
\begin{theorem}
Under the assumptions of Theorem \ref{convergence1}, let $F$ be also
$\sigma_{\alpha}$-strongly convex w.r.t. $\norm{\cdot}_{\alpha}$. For the
sequence $x^k$ generated by algorithm (RCD) we have the following rate
of convergence of the expected values of the objective function:
\begin{equation*}
   \phi^k - F^* \le \left(1-\frac{2(1-\gamma)}{N^2} \right)^k (F(x^0)-F^*),
\end{equation*}
where $\gamma$ is defined by:
\begin{equation*}
\gamma =
\begin{cases}
    1-\frac{\sigma_{\alpha}}{8 L^{1-\alpha}}, & if \ \sigma_{\alpha} \le 4 L^{1-\alpha}\\
    \frac{2L^{1-\alpha}}{\sigma_{\alpha}} , & \ otherwise.
\end{cases}
\end{equation*}
\end{theorem}

\begin{proof}
Based on relation \eqref{eqconv} it follows that:
\begin{equation*}
\begin{split}
E_{i_k j_k}[F(x^{k+1})] \le & (1-\frac{2}{N(N-1)})F(x^k) + \\
&\frac{2}{N(N-1)} \min_{d\in S}\!\left(\!f(x^k)\!+\!\langle\nabla
f(x^k),d\rangle+
L^{1-\alpha} \norm{d}^2_{\alpha}\!+\!h(x^k+d)\right).\\
\end{split}
\end{equation*}

\noindent Then, using similar derivation as in Theorem 1 we have:
\begin{equation*}
\begin{split}
&\min_{d\in S}   f(x^k)+\langle \nabla f(x^k),d\rangle+
L^{1-\alpha} \norm{d}_{\alpha}^2 + h(x^k+d)\\
&\le  \min_{y \in S} F(y) + L^{1-\alpha} \norm{y-x^k}_{\alpha}^2\\
&\le  \min_{\beta \in [0,1]} F(\beta x^* + (1-\beta)x^k) + \beta^2 L^{1-\alpha}\norm{x^k-x^*}_{\alpha}^2\\
&\le  \min_{\beta \in [0,1]} F(x^k)-\beta(F(x^k)-F^*) + \beta^2 L^{1-\alpha}\norm{x^k-x^*}_{\alpha}^2\\
&\le \min_{\beta \in [0,1]} F(x^k)+\beta\left(\frac{2\beta L^{1-\alpha}}{\sigma_{\alpha}}-1\right)
\left( F(x^k)-F^* \right),\\
\end{split}
\end{equation*}
where the last inequality results from \eqref{strong2}. The
statement of the theorem  is obtained by  noting that $\beta^*=\min
\{1,\frac{\sigma_{\alpha}}{4 L^{1-\alpha}}\}$ and the
following subcases:
\begin{enumerate}
\item If $\beta^*=\frac{\sigma_{\alpha}}{4 L^{1-\alpha}}$ and we take the expectation w.r.t. $\eta^k$ we get:
\begin{equation}\label{probrel2}
\phi^{k+1} - F^* \le  \left(1-\frac{\sigma_{\alpha}}{4 L^{1-\alpha}N^2} \right) (\phi^k-F^*),
\end{equation}
\item if $\beta^*=1$ and we take the expectation w.r.t. $\eta^k$ we get:
\begin{equation}\label{probrel3}
 \phi^{k+1} - F^* \le \left[1-\frac{2(1-\frac{2L^{1-\alpha}}{\sigma_{\alpha}})}{N^2} \right] (\phi^k -F^*).
\end{equation}
\end{enumerate}
\qed
\end{proof}


\subsection{Convergence in probability} \label{convprob}

Further, we establish some bounds on the required number of
iterations for which the generated sequence $x^k$ attains
$\epsilon$-accuracy with prespecified probability. In order to prove this result we
use Theorem 1 from \cite{RicTak:11} and for a clear understanding we
present it bellow.

\begin{lemma}\cite{RicTak:11} \label{Ric1}
 Let $\xi^0 >0 $ be a constant, $0<\epsilon<\xi^0$ and consider a nonnegative nonincreasing
sequence of (discrete) random variables $\{\xi^k\}_{k \ge 0}$ with one of the following properties:
\begin{enumerate}
 \item[(1)]
$E [\xi^{k+1} | \xi^k ] \le \xi^k -\frac{(\xi^k)^2}{c}$ for all $k$, where $c>0$ is a constant,
 \item[(2)]
$E [\xi^{k+1} | \xi^k ] \le \left(1 -\frac{1}{c} \right)\xi^k$ for all $k$ such that
$\xi^k\ge \epsilon$, where $c>1$ is a constant.
\end{enumerate}
Then, for some confidence level $\rho \in (0,1)$ we have in
probability that:
\begin{equation*}
   \text{Pr}(\xi^K \le \epsilon)\ge 1-\rho,
\end{equation*}
for a number $K$ of iterations which satisfies
\begin{equation*}
   K \ge \frac{c}{\epsilon}\left(1+ \log\frac{1}{\rho}\right)+2 - \frac{c}{\xi^0},
\end{equation*}
if  property (1) holds, or
\begin{equation*}
   K \ge c\log\frac{\xi^0}{\epsilon\rho},
\end{equation*}
if  property (2) holds.
\end{lemma}
Based on this lemma we can state the following rate of convergence
in probability:
\begin{theorem}
Let $F$ be a $\sigma_{\alpha}$-strongly convex function satisfying
Assumption \ref{assumpbeg1} and $\rho > 0$ be the confidence level.
Then, the sequence $x^k$ generated by algorithm (RCD) using uniform
distribution satisfies the following rate of convergence in
probability of the expected values of the objective function:
\begin{equation*}
   \text{Pr}(\phi^K - F^* \le \epsilon)\ge 1-\rho,
\end{equation*}
with K satisfying
\begin{equation*}
   K \ge
\begin{cases}
\frac{2N^2 L^{1-\alpha}R^2_{\alpha}}{\epsilon}\left(1+ \log\frac{1}{\rho}\right)+2 -
 \frac{2N^2 L^{1-\alpha}R^2_{\alpha}}{F(x^0)-F^*}, &\sigma_{\alpha}=0 \\
\frac{N^2}{2(1-\gamma)}\log\frac{F(x^0)-F^*}{\epsilon\rho},  &\sigma_{\alpha}>0,
\end{cases}
\end{equation*}
where
$\gamma =
\begin{cases}
    1-\frac{\sigma_{\alpha}}{8 L^{1-\alpha}}, & if \ \sigma_{\alpha} \le 4 L^{1-\alpha}\\
    \frac{2L^{1-\alpha}}{\sigma_{\alpha}} , & \ \text{otherwise}.
\end{cases}
$
\end{theorem}
\begin{proof}
 Based on relation \eqref{probrel1}, we note that taking $\xi^k$ as
$\xi^k=\phi^k-F^*$,  the property $(1)$ of Lemma \ref{Ric1} holds
and thus we get the first part of our result. Relations
\eqref{probrel2} and \eqref{probrel3} in the strongly convex case
are similar instances of property $(2)$ in Theorem \ref{Ric1} from
which we get the second part of the result. \qed
\end{proof}


\section{Generalization} \label{generalization}

In this section we study the optimization problem \eqref{model}, but
 with general linearly coupling constraints:
\begin{equation}\label{generalmodel}
\begin{split}
 \min\limits_{x\in \rset^{n}} &\ F(x) \quad \left(:=\  f(x) + h(x)\right)\\
 & \text{s.t.:}  \;\   Ax=0,
\end{split}
\end{equation}
where the functions  $f$ and $h$ have the same properties as in
Assumption \ref{assumpbeg1}  and $A \in \rset^{m \times n}$ is
a matrix with $1 < m\le n$. There are very few attempts to solve this
problem through coordinate descent strategies and up to our
knowledge the only complexity result can be found in \cite{TseYun:09}.

\noindent For the simplicity of the exposition, we work in this
section  with the standard Euclidean norm, denoted by
$\norm{\cdot}_0$, on the extended space $\rset^n$. We consider the
set of all $(m+1)$-tuples of the form $\mathcal{N}=(i^1, \dots,
i^{m+1})$, where $i^p \in \{1, \dots, N\}$ for all $p= 1, \dots,
m+1$. Also, we define $p_{\mathcal{N}}$ as the probability
distribution associated with $(m+1)$-tuples of the form
$\mathcal{N}$. Given this probability distribution
$p_{\mathcal{N}}$, for this general optimization problem
\eqref{generalmodel} we propose the following random coordinate
descent algorithm:
 \begin{equation*}
\boxed{
\begin{split}
& \text{\bf Algorithm 2 (RCD$)_{\mathcal{N}}$}\\
1.&\ \text{Choose randomly a set of $(m+1)$-tuple} \;
\mathcal{N}_k=(i^1_k, \dots, i^{m+1}_k) \\
&  \text{with probability} \ p_{\mathcal{N}_k} \\
 2.&\ \text{Set} \ x^{k+1} = x^k + d_{\mathcal{N}_k},
 \end{split}}
 \end{equation*}
where the direction $d_{\mathcal{N}_k}$ is chosen as follows:
\begin{equation*}
\begin{split}
d_{\mathcal{N}_k}=\arg \min\limits_{s \in \rset^{n}} & f(x^k) + \langle\nabla f(x^k),s\rangle +
\frac{L_{\mathcal{N}_k}}{2}\norm{s}^2_0 + h(x^k + s)\\
\text{s.t.:} \;\; & As=0, \quad s_i=0 \;\; \forall i \notin \mathcal{N}_k.\\
\end{split}
\end{equation*}
We can easily see that the linearly  coupling  constraints $ Ax =0$
prevent the development of an algorithm that performs at each
iteration a minimization with respect to less than $m+1$
coordinates. Therefore we are interested in the class of iteration
updates which restricts the objective function on feasible
directions that consist of  at least $m+1$ (block) components.

 Further, we redefine the subspace $S$ as $S=\{s \in \rset^n: \
As=0 \}$ and additionally we denote the local subspace
$S_{\mathcal{N}}= \{s \in \rset^n: \;\  As=0, \;\ s_i=0  \;\ \forall
i \in \mathcal{N}\}$. Note that we still consider an ordered
$(m+1)$-tuple $\mathcal{N}_k=(i^1_k, \dots, i^{m+1}_k)$ such that
$i^p_k \neq i^l_k$ for all $p \neq l$. We observe that for a general
matrix $A$, the previous subproblem does not necessarily have a
closed form solution. However, when $h$ is coordinatewise separable,
strictly convex and piece-wise linear/quadratic with
$\mathcal{O}(1)$ pieces (e.g., $h$ given by \eqref{norm1}) there are
efficient algorithms for solving the previous subproblem in
linear-time \cite{TseYun:09}. Moreover, when $h$ is the box indicator
function (i.e., $h$ given by \eqref{indicator}) we have the
following: in the scalar case (i.e., $N=n$) the subproblem has a
closed form solution; for  the block case (i.e., $N<n$) there exist
linear-time algorithms for solving the subproblem within
$\mathcal{O}(\sum_{i \in \mathcal{N}_k} n_i)$ operations
\cite{BerKov:93}. Through a similar reasoning as in Lemma \ref{lemmalij}
we can derive that given a set of indices $ \mathcal{N}=( i^1,
\dots, i^p)$, with $p \geq 2$, the following relation holds:
\begin{equation}\label{lipsgen}
f(x+d_{\mathcal{N}}) \le f(x) + \langle \nabla f(x),
d_{\mathcal{N}}\rangle +
\frac{L_{\mathcal{N}}}{2}\norm{d_{\mathcal{N}}}^2_0,
\end{equation}
for all $x \in \rset^n$ and  $d_{\mathcal{N}} \in \rset^n$ with
nonzero entries only on the blocks $i^1, \dots, i^p$. Here,
$L_{\mathcal{N}}= L_{i^1}+ \dots + L_{i^p}$.  Moreover, based on
Lemma \ref{lemmaconf} it follows that  any $d \in S$ has a conformal
realization defined by $d=\sum_{t=1}^{s} d^t$, where the elementary
vectors $d^t \in S$ are conformal to $d$ and have at most $m+1$
nonzeros. Therefore, any vector $d \in S$ can be generated by
$d=\sum_{\mathcal{N}} s_{\mathcal{N}}$, where the vectors
$s_{\mathcal{N}} \in S_{\mathcal{N}}$  have at most $m+1$ nonzero
blocks and are conformal to $d$. We now present the main convergence
result for this method.


\begin{theorem}
\label{convergence2}
 Let $F$ satisfy Assumption \ref{assumpbeg1}.
Then, the random coordinate descent algorithm (RCD)$_{\mathcal{N}}$
that chooses uniformly at each iteration $m+1$ blocks generates a
sequence $x^k$ satisfying the following rate of convergence for the
expected values of the objective function:
\begin{equation*}
   \phi^k - F^* \le \frac{N^{m+1} LR_{0}^2 }{k+\frac{N^{m+1} LR_0^2}{F(x^0)-F^*}}.
\end{equation*}
\end{theorem}
\begin{proof}
The proof is similar to that of Theorem \ref{convergence1} and we
omit it here for brevity.
\end{proof}


\section{Complexity analysis}
\label{complanalysis} In this section we analyze the total
complexity (arithmetic complexity \cite{Nes:04}) of  algorithm (RCD)
based on extended Euclidean norm for optimization problem
\eqref{model} and compare it with other complexity estimates. Tseng
presented in \cite{TseYun:09} the first complexity bounds for the (CGD)
method applied to our optimization problem \eqref{model}. Up to our
knowledge there are no other complexity results for coordinate
descent methods on the general optimization model \eqref{model}.

Note that the algorithm (RCD)  has an overall complexity w.r.t.
extended Euclidean norm  given by:
\begin{equation*}
 \mathcal{O}\left( \frac{N^2 L R^2_0}{\epsilon} \right)\mathcal{O}(i_{RCD}),
\end{equation*}
where $\mathcal{O}(i_{RCD})$ is the complexity per iteration of
algorithm (RCD). On the other hand, algorithm (CGD) has the
following complexity estimate:
\begin{equation*}
 \mathcal{O}\left( \frac{n L R^2_0}{\epsilon} \right)\mathcal{O}(i_{CGD}),
\end{equation*}
where $\mathcal{O}(i_{CGD})$ is the iteration complexity  of
algorithm (CGD). Based on the particularities and computational
effort of each method, we will show in the sequel that for some
optimization models arising  in real-world  applications the
arithmetic complexity of (RCD) method is lower than that of (CGD)
method. For certain instances of problem \eqref{model} we have that
the computation of the coordinate directional derivative of the
smooth component of the objective function is much more simpler than
 the  function evaluation  or directional derivative along
 an arbitrary direction. Note that the iteration of algorithm (RCD) uses
only a small number of coordinate directional derivatives of the
smooth part of the objective, in contrast with the (CGD) iteration
which requires the full gradient. Thus, we estimate the arithmetic
complexity of these two methods applied to a class of optimization
problems containing instances for which the directional derivative
of objective function can be computed cheaply. We recall that the
process of choosing a uniformly random pair $(i,j)$ in our method
requires $\mathcal{O}(1)$ operations.

Let us structure a general coordinate descent iteration in two
phases:\\
{\bf Phase 1}: Gather first-order information to form a quadratic
approximation of the original optimization problem. \\
{\bf Phase 2}: Solve a quadratic optimization problem using data
acquired at Phase 1 and update the current vector.\\
Both algorithms (RCD) and (CGD) share this structure but, as we will
see, there is a gap between computational complexities. We analyze
the following example:
\begin{equation}\label{genfunc}
 f(x)=\frac{1}{2} x^T Z^T Z x + q^T x,
\end{equation}
where $Z= \left[z_1 \ \dots \ z_n\right] \in \rset^{m \times n}$ has
sparse columns, with an average $p << n$ nonzero entries on each
column $z_i$ for all $i=1, \dots, n$. A particular case of this
class of functions is $f(x)=\frac{1}{2}\norm{Zx-q}^2$, which has
been considered for numerical experiments in \cite{Nes:10} and
\cite{RicTak:11}. The problem \eqref{model}, with the aforementioned
structure \eqref{genfunc} of the smooth part of the objective
function,  arises in many applications: e.g., linear SVM
\cite{TseYun:07}, truss topology \cite{NesShp:12}, internet (Google problem)
\cite{Nes:10}, Chebyshev center problems \cite{XuFre:03},  etc. The
reader can easily find many other examples of optimization problems
with cheap coordinate directional derivatives.

Further, we estimate the iteration complexity of the algorithms (RCD) and (CGD).
Given a feasible $x$, from the expression
\begin{equation*}
\nabla_i f(x) = \langle z_i, Zx \rangle + q^i,
\end{equation*}
we note that if the residual $r(x)=Zx$ is already known, then the
computation of $\nabla_i f(x)$ requires $\mathcal{O}(p)$ operations.
We consider that the dimension $n_i$ of each block is of order
$\mathcal{O}(\frac{n}{N})$. Thus, the (RCD) method updates the
current point $x$ on $\mathcal{O}(\frac{n}{N})$ coordinates and
summing up with the computation of the new residual $r(x^+)=Zx^+$,
which in this case requires $\mathcal{O}(\frac{pn}{N})$ operations,
we conclude that up to this stage, the iteration of (RCD) method has
numerical complexity  $\mathcal{O}(\frac{pn}{N})$. However,  the
(CGD) method requires the computation of the  full gradient for
which are necessary $\mathcal{O}(n p)$ operations. As a preliminary
conclusion, Phase 1 has the following complexity regarding the two
algorithms:
\begin{equation*}
\begin{split}
& \text{\it (RCD). Phase 1}: \ \ \mathcal{O} (\frac{np}{N})  \\
& \text{\it (CGD). Phase 1}: \ \ \mathcal{O} (np)
\end{split}
\end{equation*}

Suppose now that for a given $x$, the blocks
$(\nabla_{i}f(x),\nabla_{j}f(x))$ are known for (RCD) method or the
entire gradient vector $\nabla f(x)$ is available for (CGD) method
within previous computed complexities, then the second phase
requires the finding of an update direction with respect to each
method. For the general linearly constrained model \eqref{model},
evaluating the iteration complexity of both algorithms can be a
difficult task. Since in \cite{TseYun:09} Tseng provided an explicit
total computational complexity for the cases when the nonsmooth part
of the objective function $h$ is separable and piece-wise
linear/quadratic with $\mathcal{O}(1)$ pieces, for clarity of the
comparison we also analyze the particular setting when $h$ is a box
indicator function as given in equation \eqref{indicator}. For
algorithm (RCD) with $\alpha=0$, at each iteration, we require the
solution of the following problem (see \eqref{indicator}):
\begin{equation}
\label{probiter}
\begin{split}
& \min_{s_{ij} \in \rset^{n_{i}+n_{j}}} \langle \nabla_{ij} f(x),
s_{ij}\rangle
+ \frac{L_{ij}^0}{2}\norm{s_{ij}}^2_0\\
& \text{s.t.:}   \;\;\;  a_i^Ts_i+a_j^Ts_j=0,   \;\;\ (l-x)_{ij} \le
s_{ij} \le (u-x)_{ij}.
\end{split}
\end{equation}
It is shown in \cite{Kiw:07} that problem \eqref{probiter} can be
solved in $\mathcal{O}(n_{i}+n_{j})$ operations. However, in the
scalar case (i.e., $N=n$) problem \eqref{probiter} can solved in
closed form. Therefore, Phase 2 of algorithm (RCD)  requires
$\mathcal{O}(\frac{n}{N})$ operations. Finally, we estimate for
algorithm (RCD) the total arithmetic complexity in terms of the
number of blocks $N$ as:
\begin{equation*}
   \mathcal{O}\left( \frac{N^2 L R_0^2}{\epsilon} \right)\mathcal{O}(\frac{pn}{N}).
\end{equation*}

On the other hand, due to the  Gauss-Southwell rule, the (CGD)
method requires at each iteration the solution of a  quadratic
knapsack problem of dimension $n$. It is argued in \cite{Kiw:07} that
for solving the quadratic knapsack problem we need $\mathcal{O}(n)$
operations. In conclusion, the Gauss-Southwell procedure in
algorithm (CGD) requires the conformal realization of the solution
of a continuous knapsack problem and the selection of a
 ``good'' set of blocks $\mathcal{J}$. This last process has a
different cost depending on  $m$. Overall, we estimate the total
complexity of algorithm (CGD) for one equality constraint, $m=1$,
as:
\begin{equation*}
\mathcal{O}\left(\frac{n L R_0^2}{\epsilon}\right)\mathcal{O}(pn)
\end{equation*}

\vspace*{-0.8cm}
 \setlength{\extrarowheight}{5pt}
\begin{table}[ht]
\centering \caption{Comparison of  arithmetic complexities for
alg. (RCD), (CGD) and \cite{HusKel:06,LisSim:05} for $m=1$.}
\begin{tabular}{| c | c | c | c |}
\hline
Algorithm / $m=1$ & $h(x)$ & Probabilities & Complexity\\ [1ex]
\hline
\text{\bf (RCD)} & separable & $\frac{1}{N^2}$   & $\mathcal{O}(\frac{pNn LR_0^2}{\epsilon})$ \\[1ex]
\hline
\text{\bf (CGD)}& separable& greedy &$\mathcal{O}(\frac{pn^2 LR_0^2}{\epsilon})$ \\[1ex]
\hline
\text{\bf Hush \cite{HusKel:06}, List \cite{LisSim:05}} & box indicator & greedy & $\mathcal{O}(\frac{pn^2 LR_0^2}{\epsilon})$\\[1ex]
\hline
\end{tabular}
\end{table}

First, we note that in the case  $m=1$ and $N<<n$ (i.e., the block
case) algorithm (RCD) has better arithmetic complexity than
algorithm (CGD) and previously mentioned block-coordinate methods
\cite{HusKel:06,LisSim:05} (see Table 1). When $m=1$ and  $N=n$ (i.e.,
the scalar case),  by  substitution in the above expressions from
 Table 1, we have a total complexity for algorithm (RCD) comparable
to the complexity  of algorithm (CGD) and the algorithms from
\cite{HusKel:06,LisSim:05}.

On the other hand,  the complexity of choosing a random pair $(i,j)$
in  algorithm (RCD) is very low, i.e., we need  $\mathcal{O}(1)$
operations. Thus, choosing the working pair  $(i,j)$ in our
algorithm (RCD)  is much simpler than choosing the working set
$\mathcal{J}$ within the Gauss-Southwell rule for algorithm (CGD)
which assumes the following steps: first, compute the  projected
gradient direction and second, find the conformal realization of
computed direction; the overall complexity of these two steps being
$\mathcal{O}(n)$. In conclusion, the algorithm (RCD) has a very
simple implementation due to simplicity of the random choice for the
working pair and a low complexity per iteration.

For the case  $m=2$ the algorithm (RCD) needs in Phase 1 to compute
coordinate directional derivatives with complexity
$\mathcal{O}(\frac{pn}{N})$ and in Phase 2 to find the solution of a
3-block dimensional problem of the same structure as
\eqref{probiter} with complexity $\mathcal{O}(\frac{n}{N})$.
Therefore, the iteration complexity of the (RCD) method in this case
is still $\mathcal{O}(\frac{pn}{N})$. On the other hand, the
iteration complexity of the algorithm (CGD) for $m=2$ is given by
$\mathcal{O}(pn + n\log n)$ \cite{TseYun:09}.

For $m>2$,  the complexity of Phase 1 at each iteration of our
method still requires $\mathcal{O}(\frac{pn}{N})$ operations and the
complexity of Phase 2 is $\mathcal{O}(\frac{mn}{N})$, while in the
(CGD) method the iteration complexity is $\mathcal{O}(m^3n^2)$
\cite{TseYun:09}.

For the case $m>1$,  a comparison  between arithmetic complexities
of algorithms (RCD) and (CGD) is provided in Table 2. We see from
this table that depending on the values of $n, m$ and $N$, the
arithmetic complexity of (RCD) method can be better or worse than
that  of  the (CGD) method.

\setlength{\extrarowheight}{5pt}
\begin{table}[ht]
\centering
\caption{Comparison of  arithmetic complexities for
 algorithms (RCD) and (CGD) for $m \geq 2$.}
\begin{tabular}{| c | c | c |}
\hline
Algorithm & $m=2$ & $m>2$ \\ [1ex]
\hline
(RCD) & $\frac{pN^2nLR^2_0}{\epsilon}$ & $\frac{(p+m)N^m nLR^2_0}{\epsilon}$\\[1ex]
\hline
(CGD) & $\frac{(p+\log n)n^2 LR^2_0}{\epsilon}$    & $\frac{m^3 n^3 LR^2_0}{\epsilon}$\\[1ex]
\hline
\end{tabular}
\end{table}

We conclude from the rate of convergence and the previous complexity
analysis that algorithm (RCD) is easier to be implemented and
analyzed due to the randomization and the typically  very simple
iteration. Moreover, on certain classes of problems with sparsity
structure, that appear frequently in many large-scale real
applications, the arithmetic complexity of (RCD) method is better
than that of some well-known methods from the literature.  All these
arguments make the algorithm (RCD) to be competitive in the
composite optimization framework. Moreover, the (RCD) method is
suited for recently developed computational architectures (e.g.,
distributed or parallel architectures).


\section{Numerical Experiments} \label{numexp}

In this section we present extensive numerical simulations, where we
compare our algorithm (RCD) with some recently developed
state-of-the-art algorithms from the literature for solving the
optimization problem \eqref{model}: coordinate gradient descent
(CGD)  \cite{TseYun:09}, projected gradient method for composite
optimization (GM)  \cite{Nes:07} and LIBSVM  \cite{ChaLin:11}. We
tested  the four methods on  large-scale optimization problems
ranging from $n = 10^3$ to $n = 10^7$ arising in various
applications such as: support vector machine (SVM) (Section 6.1),
the  Chebyshev center of a set of points (Section 6.2) and random
generated problems with an $\ell_1$-regularization term (Section
6.3). Firstly, for the SVM application,  we compare algorithm (RCD)
against (CGD) and LIBSVM and we remark that our algorithm has the
best performance on large-scale problem instances with sparse data.
Secondly, we also observe  a more robust behavior for algorithm
(RCD) in comparison with algorithms (CGD) and (GM) when using
different initial points on Chebyshev center problem instances.
Lastly, we tested  our algorithm on randomly generated problems,
where the nonsmooth part of the objective function contains an
$\ell_1$-norm term, i.e., $\lambda \sum_{i=1}^n |x_i|$ for some
$\lambda>0$, and we compared our method with algorithms (CGD) and
(GM).

We have implemented all the algorithms in C-code  and the
experiments were run on a PC with an Intel Xeon E5410 CPU and 8 GB
RAM memory. In all algorithms we considered the scalar case, i.e.,
$N=n$ and we worked with the extended Euclidean norm ($\alpha = 0$).
In our applications the smooth part $f$ of the composite objective
function is of the form \eqref{genfunc}.  The coordinate directional
derivative at the current point for algorithm (RCD) $\nabla_i
f(x)=\langle z_i, Zx \rangle +q_i$  is computed efficiently by
knowing at each iteration the residual $r(x)=Zx$. For the (CGD)
method, the working set is chosen accordingly to Section 6 in
\cite{TseYun:07}. Therefore, the entire gradient at the current
point, $\nabla f(x)= Z^T Z x +q $, is required, which is computed
efficiently using the residual $r(x)=Zx$. For gradient and residual
computations we used  an efficient sparse matrix-vector
multiplication procedure. We coded the standard (CGD) method
presented in \cite{TseYun:09} and we have not used any heuristics
recommended by Tseng in \cite{TseYun:07}, e.g., the ``3-pair''
heuristic technique. The direction $d_{ij}$ at the current point
from subproblem \eqref{extdir} for algorithm (RCD)  is computed in
closed form for all three applications considered in this section.
For computing the direction $d_H(x;\mathcal{J})$ at the current
point from subproblem \eqref{CGDiter} in the (CGD) method for the
first two applications we coded the algorithm from \cite{Kiw:07} for
solving quadratic knapsack problems of the form \eqref{probiter}
that has linear time complexity.  For the second application, the
direction at the current point for algorithm (GM) is computed using
a linear time simplex projection algorithm introduced in
\cite{JudRay:08}. For the
 third application, we used the equivalent formulation  of the subproblem
\eqref{CGDiter}  given in \cite{TseYun:09}, obtaining for both
algorithms (CGD) and (GM) an iteration which requires the solution
of some double size quadratic knapsack problem of the
form~\eqref{probiter}.

In the following tables we present for each algorithm  the final
objective function value (obj), the number of iterations (iter) and
the necessary CPU time  for our computer to execute all the
iterations. As the algorithms (CGD),  LIBSVM and (GM) use the whole
gradient information to obtain the working set and to find the
direction at the current point, we also report for the algorithm
(RCD) the equivalent number of full-iterations which means the total
number of iterations divided by $\frac{n}{2}$ (i.e., the number of
iterations groups $x^0, x^{n/2}, \dots, x^{k n/2 }$).


\subsection{Support vector machine}
In order to better understand the practical performance of our
method, we have tested the algorithms (RCD), (CGD) and LIBSVM on
two-class data classification problems with linear kernel, which is
a well-known real-world application that can be posed as a
large-scale optimization problem in the form \eqref{model} with a
sparsity structure. In this section, we describe our implementation
of algorithms (RCD), (CGD) \cite{TseYun:07} and LIBSVM \cite{ChaLin:11} and
report the numerical results on different test problems. Note that
linear SVM is a technique mainly used for text classification, which
can be formulated as the following optimization problem:
\begin{equation}\label{modelnum}
\begin{split}
 \min\limits_{x\in \rset^n} \ &\frac{1}{2}x^TZ^TZx - e^Tx + \mathbf{1}_{[0,C]}(x)\\
\text{s.t.:}  & \;\;  a^Tx=0, \\
\end{split}
\end{equation}
where $\mathbf{1}_{[0,C]}$ is the indicator function for the box
constraint  set $[0,C]^n$, $Z \in \rset^{m \times n}$ is the
instance matrix with an average sparsity degree $p$ (i.e.,  on
average, $Z$ has $p$ nonzero entries on each column), $a \in
\rset^n$ is the label vector of instances, $C$ is the penalty
parameter and $e= [1 \dots 1]^T \in \rset^n$. Clearly, this model
fits the aforementioned class of functions \eqref{genfunc}.  We set
the primal penalty parameter $C=1$ in all SVM test  problems. As in
\cite{TseYun:07}, we initialize all the algorithms with $x^0=0$. The
stopping criterion used in the algorithm (RCD) is:
$f(x^{k-j})-f(x^{k-j+1})\le \epsilon$, where
 $j=0, \dots, 10$, while for the algorithm (CGD)
we use the stopping criterion $f(x^k) - f(x^{k+1})\le \epsilon$,
where $\epsilon=10^{-5}$.

\setlength{\extrarowheight}{5pt}
\begin{table}[ht]
\centering \caption{Comparison of algorithms (RCD), (CGD) and
library LIBSVM on SVM problems.} {\tiny
\begin{tabular}{| p{0.6cm} | p{1.2cm} | p{2.5cm} | p{2.3cm} | p{2.3cm} |}
\hline
Data set & $n/m$ & (RCD) & (CGD) &LIBSVM\\
\hline
& &full-iter/obj/time(min)  &    iter/obj/time(min)  &   iter/obj/time(min)  \\
\hline
\hline
a7a&16100/122 $(p=14)$&11242/-5698.02/2.5& 23800/-5698.25/21.5&63889/-5699.25/0.46\\
\hline
a8a & 22696/123 $(p=14)$& 22278/-8061.9/18.1 & 37428/-8061.9/27.8 & 94877/-8062.4/1.42\\
\hline
a9a  & 32561/123 $(p=14)$&15355/-11431.47/7.01&45000/-11431.58/89 &78244/-11433.0/2.33\\
\hline
w8a & 49749/300 $(p=12)$& 15380/-1486.3/26.3 & 19421/-1486.3/27.2 & 130294/-1486.8/42.9\\
\hline
ijcnn1 & 49990/22 $(p=13)$&7601/-8589.05/6.01&9000/-8589.52/16.5&15696/-8590.15/1.0\\
\hline
web & 350000/254 $(p=85)$& 1428/-69471.21/29.95 & 13600/-27200.68/748 & 59760/-69449.56/467 \\
\hline
covtyp  & 581012/54 $(p=12)$&1722/-337798.34/38.5& 12000/-24000/480 & 466209/-337953.02/566.5 \\
\hline
test1 & $2.2~\cdot~10^6/10^6$ $(p=50)$&  228/-1654.72/51 & 4600/-473.93/568 & * \\
\hline
test2 & $10^7/5\cdot 10^3$ $(p=10)$ & 350/-508.06/112.65 & 502/-507.59/516.66 & * \\
\hline
\end{tabular}
}
\end{table}

We report in Table $3$ the results for algorithms (RCD), (CGD) and
LIBSVM implemented in the scalar case, i.e., $N=n$. The data used
for the experiments can be found on the LIBSVM webpage
(http://www.csie.ntu.edu.tw/cjlin/libsvmtools/ datasets/). For
problems with very large dimensions, we generated the data randomly
(see ``test1'' and ``test2'')  such that the nonzero elements  of
$Z$ fit into the available memory of our computer. For each
algorithm we present the final objective function value (obj), the
number of iterations (iter) and the necessary CPU time (in minutes)
 for our computer to execute all the iterations.  For the
algorithm (RCD) we report the equivalent number of full-iterations,
that is  the number of iterations groups $x^0, x^{n/2}, \dots, x^{k
n/2 }$. On small test problems we observe that LIBSVM outperforms
algorithms (RCD) and (CGD), but we still have that the CPU time for
algorithm (RCD) does not exceed $30$ min, while algorithm (CGD)
performs much worse. On the other hand, on large-scale problems the
algorithm (RCD) has the best behavior among the three tested
algorithms (within a factor of $10$). For very large  problems ($n
\geq 10^6$), LIBSVM has not returned any result within $10$ hours.

\begin{figure}[ht]
\centering \caption{Performance of algorithm (RCD) for different
block dimensions.}
\includegraphics[width=6cm]{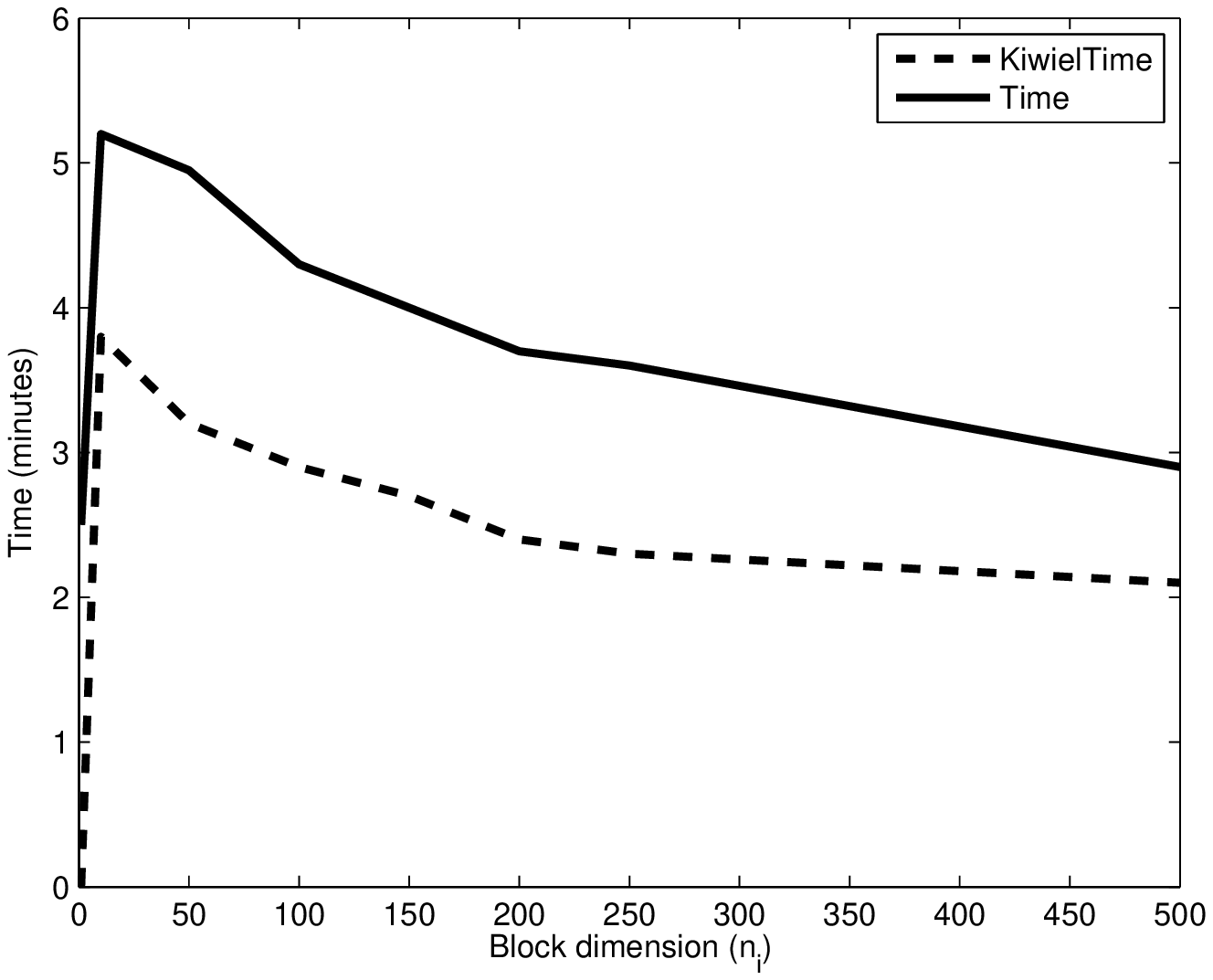}
\includegraphics[width=6cm]{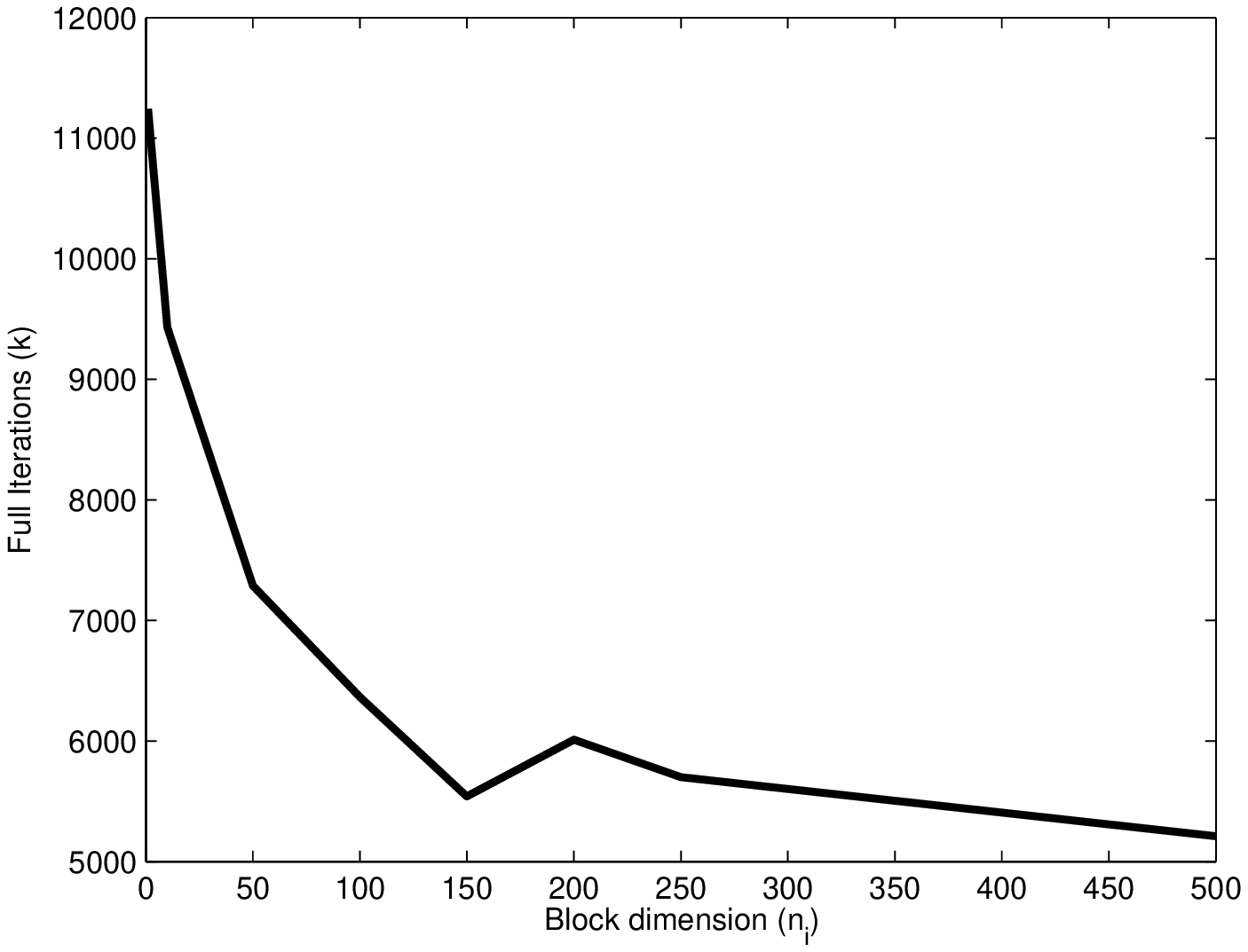}
\end{figure}

For the block case (i.e., $N \leq n$), we have plotted for algorithm
(RCD) on the test problem ``a7a'' the  CPU time and total time (in
minutes) to solve knapsack problems (left) and the number of
full-iterations (right) for different dimensions of the blocks
$n_i$. We see that the number of iterations decreases with the
increasing dimension of the blocks, while the CPU time increases
w.r.t. the scalar case due to the fact that for $n_i > 1$ the
direction $d_{ij}$ cannot be computed in closed form as in the
scalar case (i.e., $n_i=1$), but requires solving a quadratic
knapsack problem  \eqref{probiter} whose solution can be computed in
$\mathcal{O}(n_{i}+n_{j})$ operations \cite{Kiw:07}.


\subsection{Chebyshev center of a set of points}
Many real applications such as location planning of shared
facilities, pattern recognition, protein analysis, mechanical
engineering and computer graphics (see e.g., \cite{XuFre:03} for
more details and appropriate references) can be formulated as
finding the Chebyshev center of a given set of points.  The
Chebyshev center problem involves the following: given a set of
points $z_1, \dots, z_n \in \rset^m$, find the center $z_c$ and
radius $r$ of the smallest enclosing ball of the given points. This
geometric problem can be formulated as the following optimization
problem:
\begin{align*}
& \min\limits_{r, z_c} \;  r\\
& \text{s.t.:}  \;\;\   \norm{z_i - z_c}^2 \le r \quad \forall i=1,
\dots, n,
\end{align*}
where $r$ is the radius  and $z_c$ is the center  of the enclosing
ball. It can be immediately seen that the dual formulation of this
problem is a particular case of our linearly constrained
optimization model \eqref{model}:
\begin{align}
& \min\limits_{x \in \rset^n} \norm{Zx}^2 - \sum\limits_{i=1}^n \norm{z_i}^2x_i  + \mathbf{1}_{[0,\infty)}(x) \label{problem3}\\
& \text{s.t.:}  \;\;\   e^T x = 1,  \nonumber
\end{align}
where $Z$ is the matrix containing the given points $z_i$ as
columns. Once an optimal solution $x^*$ for the dual formulation is
found, a primal solution can be recovered as follows:
\begin{equation}\label{radcent}
    r*= \left( -\norm{Zx^*}^2 + \sum\limits_{i=1}^n \norm{z_i}^2 x_i^* \right)^{1/2},  \qquad  z_c^* = Z x^*.
\end{equation}

\begin{figure}[h]
\label{fig1} \centering \caption{Performance of algorithms (RCD),
(GM) and (CGD) for 50 full-iterations and initial point $e_1$ (top)
and $\frac{e}{n}$ (bottom) on a randomly generated matrix  $Z \in
\rset^{2 \times 1000}$.}
\includegraphics[width=4cm]{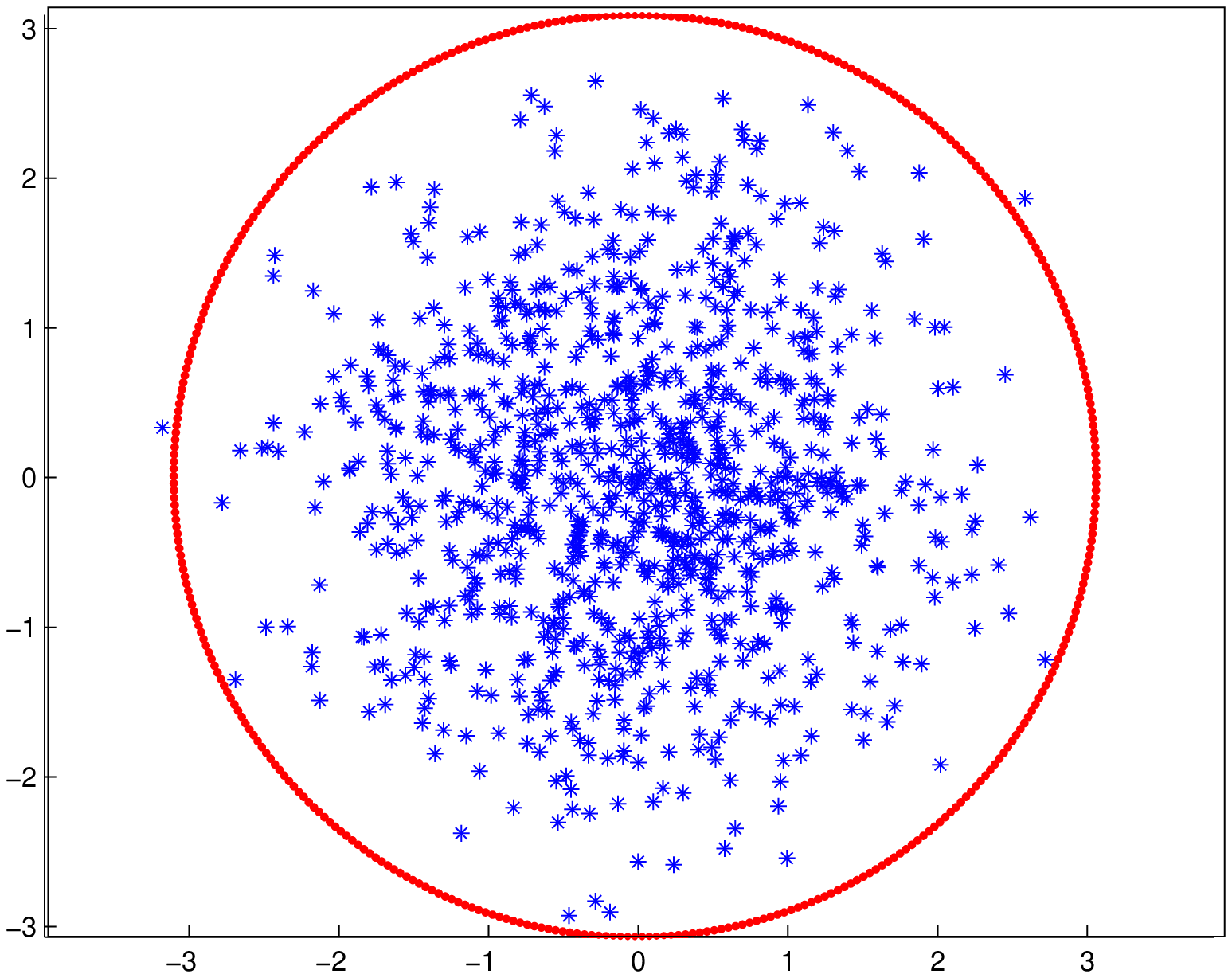}
\includegraphics[width=4cm]{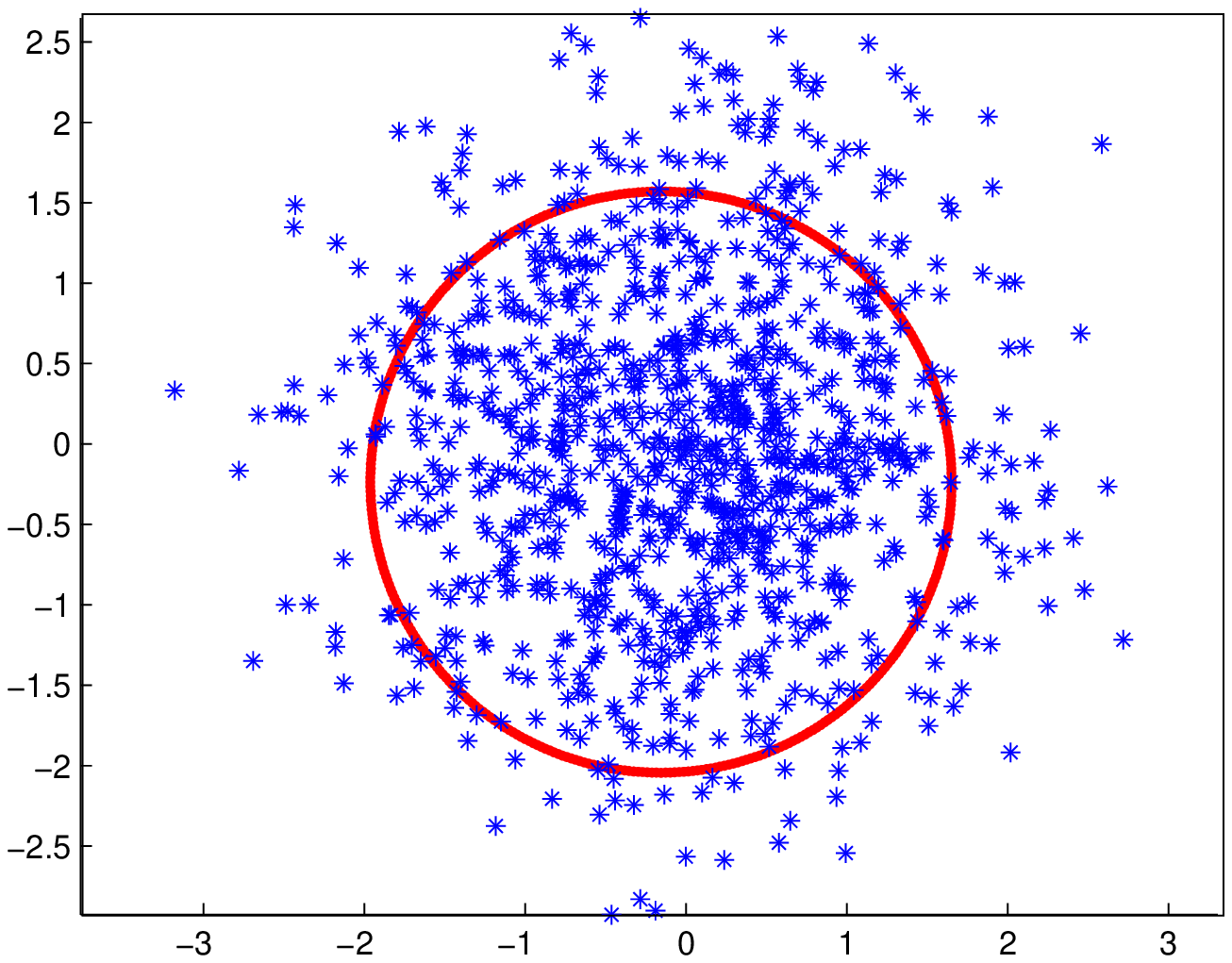}
\includegraphics[width=4cm]{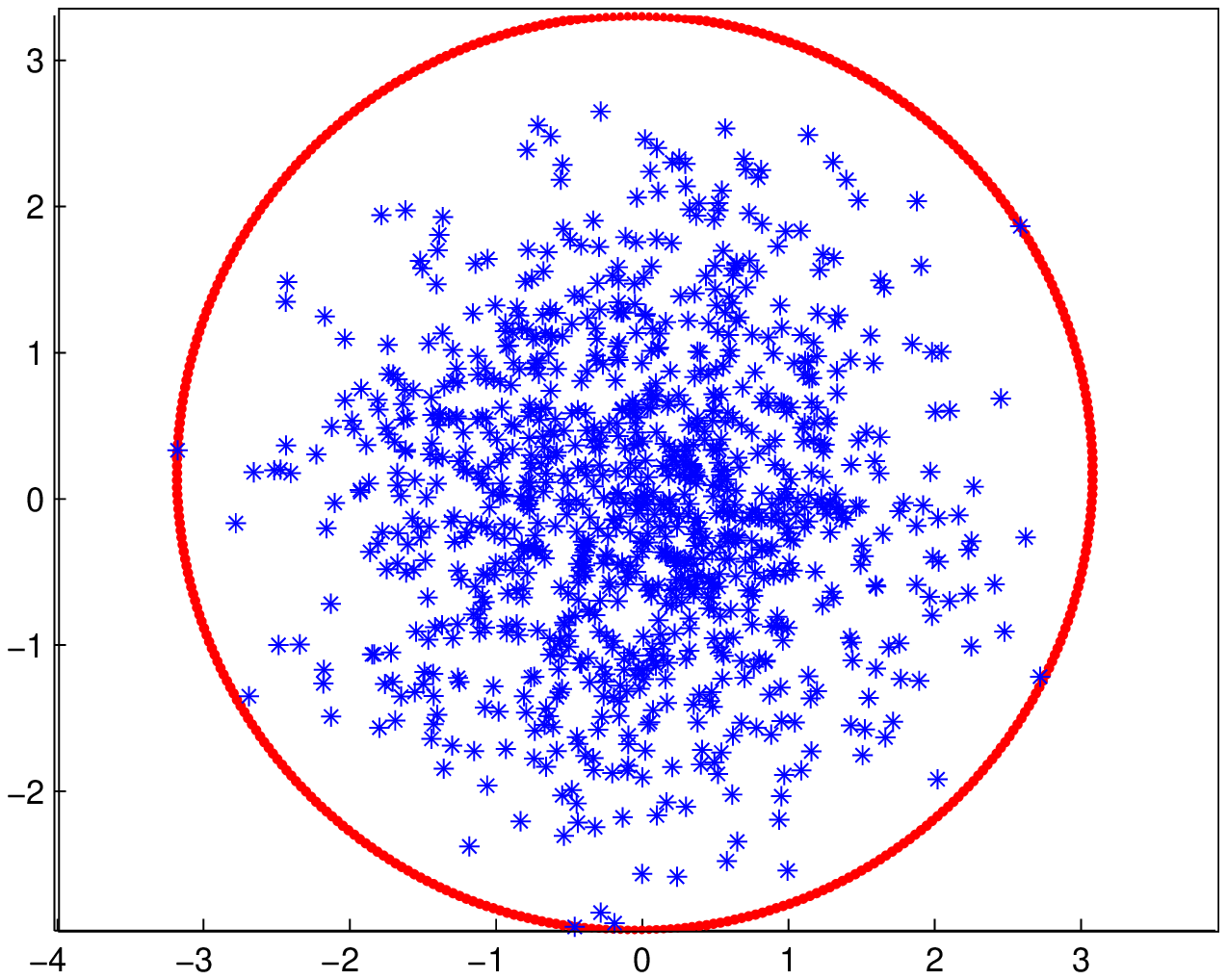}
\includegraphics[width=4cm]{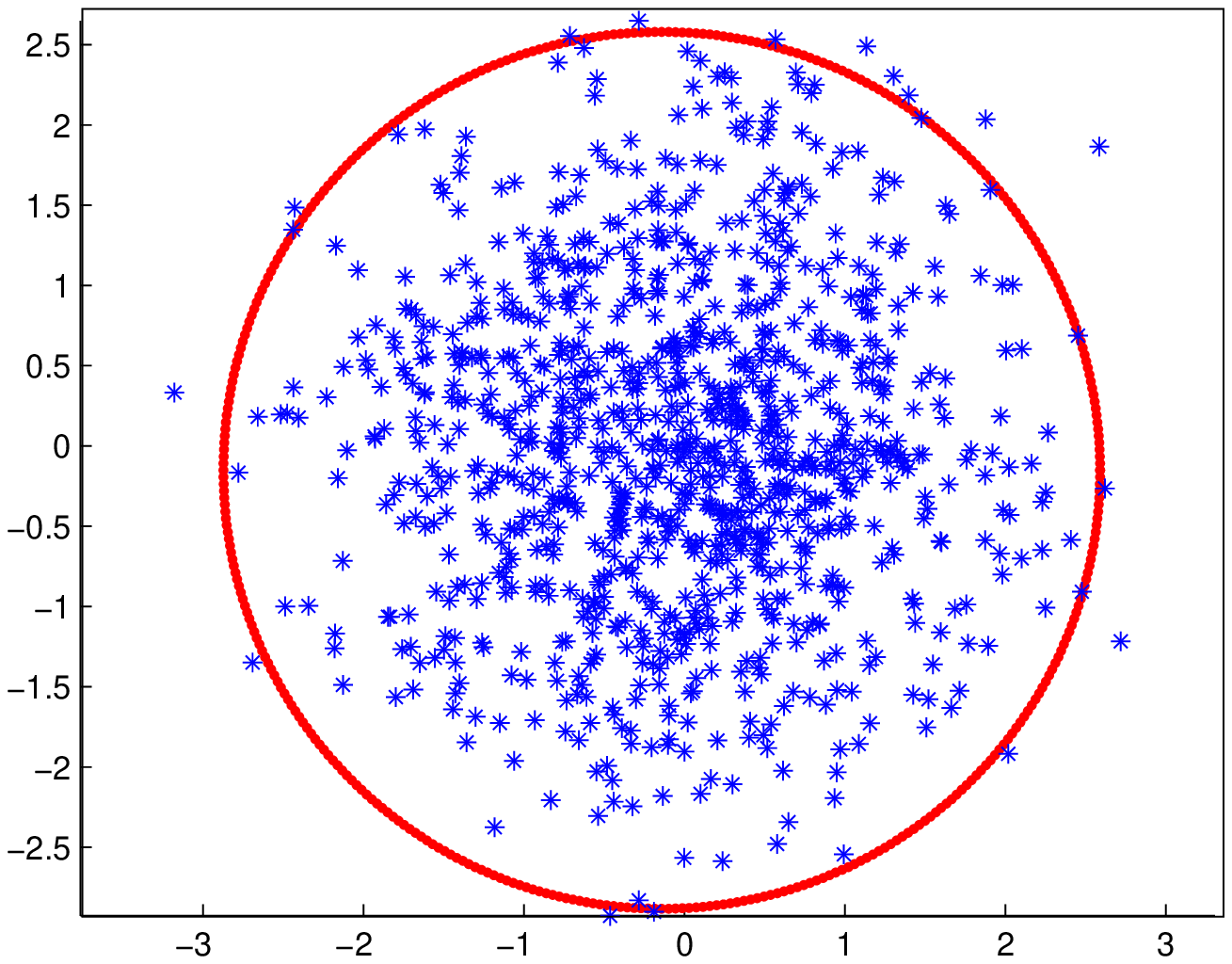}
\includegraphics[width=4cm]{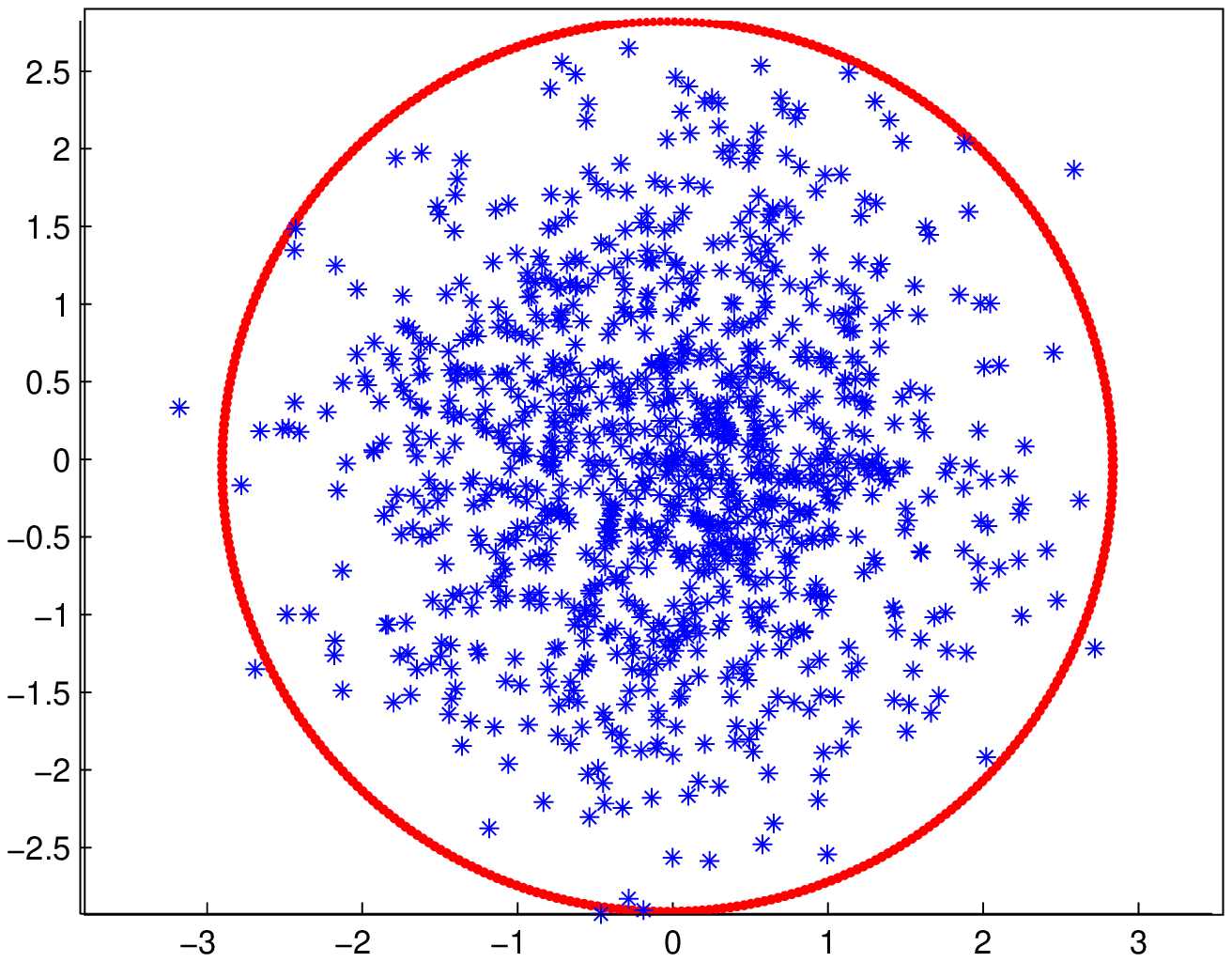}
\includegraphics[width=4cm]{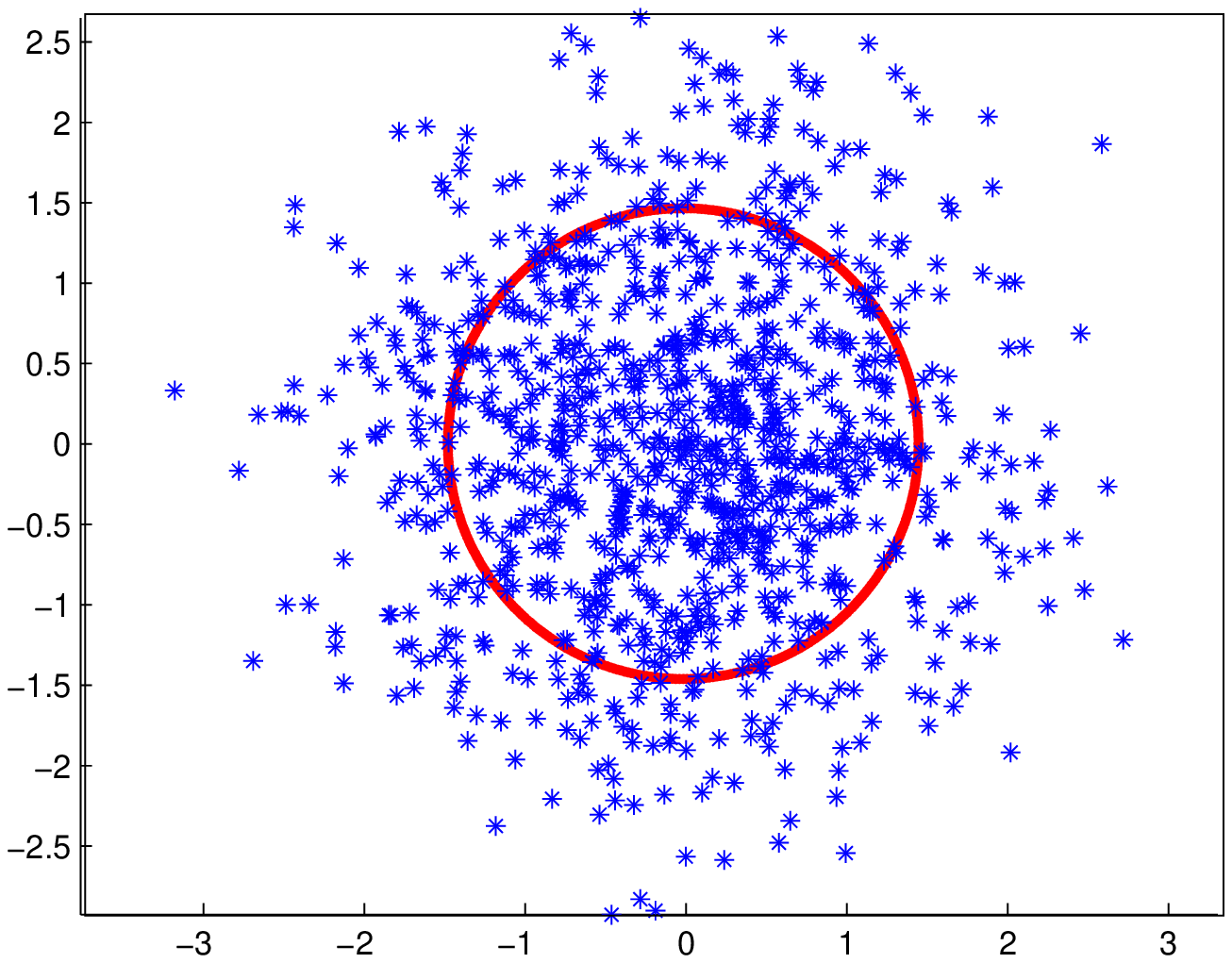}
RCD \hspace{3.5cm} GM \hspace{3.5cm} CGD
\end{figure}

The direction $d_{ij}$  at the current point in the algorithm (RCD)
is computed in closed form. For computing the direction in the (CGD)
method  we need to solve a quadratic knapsack problem that has
linear time complexity \cite{Kiw:07}.  The direction at the current
point for algorithm (GM) is computed using a linear time simplex
projection algorithm introduced in \cite{JudRay:08}. We compare
algorithms (RCD), (CGD)  and (GM) for a set of large-scale problem
instances generated randomly with a uniform distribution. We recover
a suboptimal radius and Chebyshev center using the same set of
relations \eqref{radcent} evaluated at the final iteration point
$x^{k}$ for all three algorithms.

\begin{figure}[t]
\label{fig2} \centering \caption{Time performance of algorithms
(RCD), (GM) and (CGD) for initial point $\frac{e}{n}$ (left) and
$e_1$(right) on a randomly generated matrix  $Z \in \rset^{30 \times
1000}$. }
\includegraphics[width=6cm]{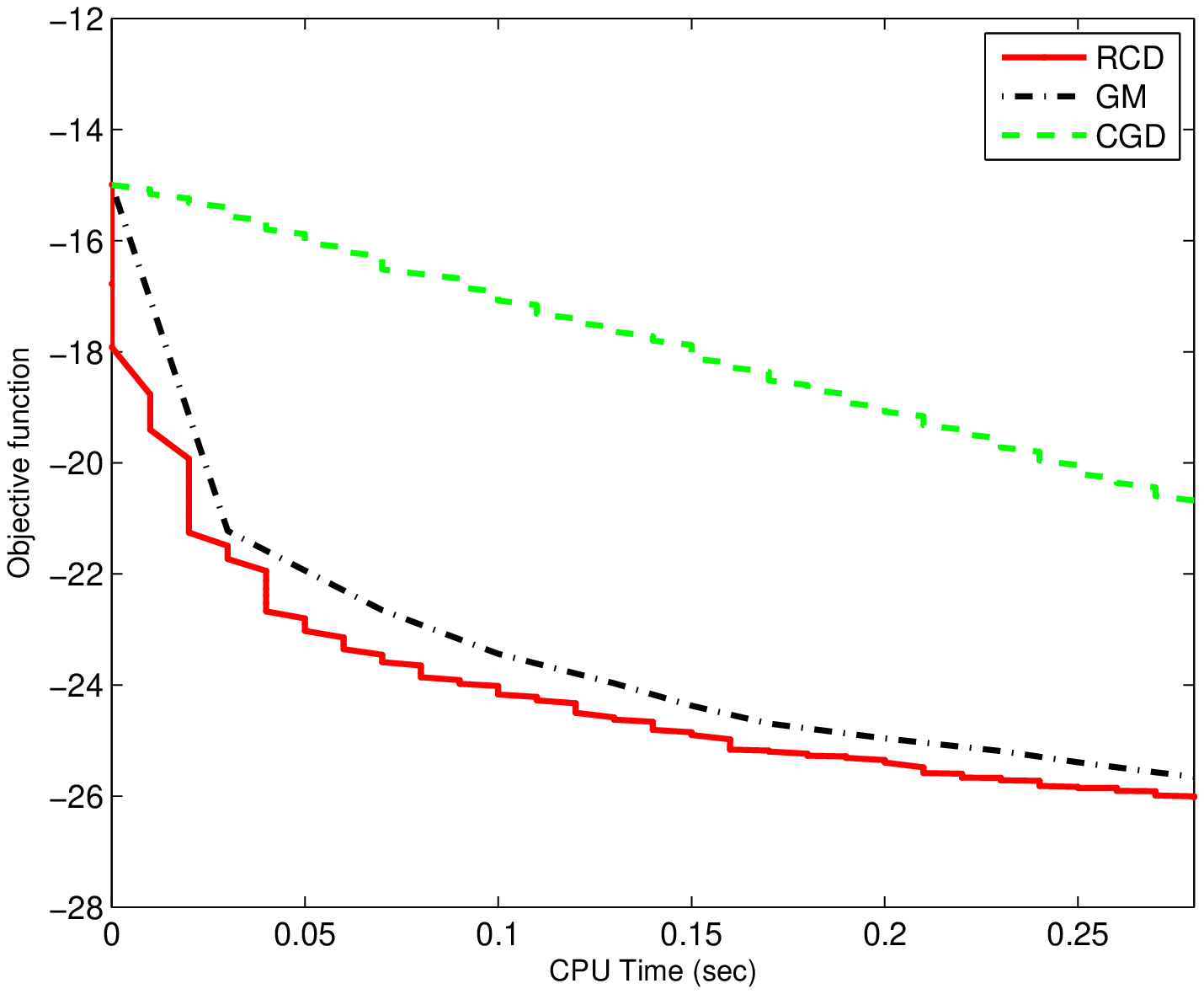}
\includegraphics[width=6cm]{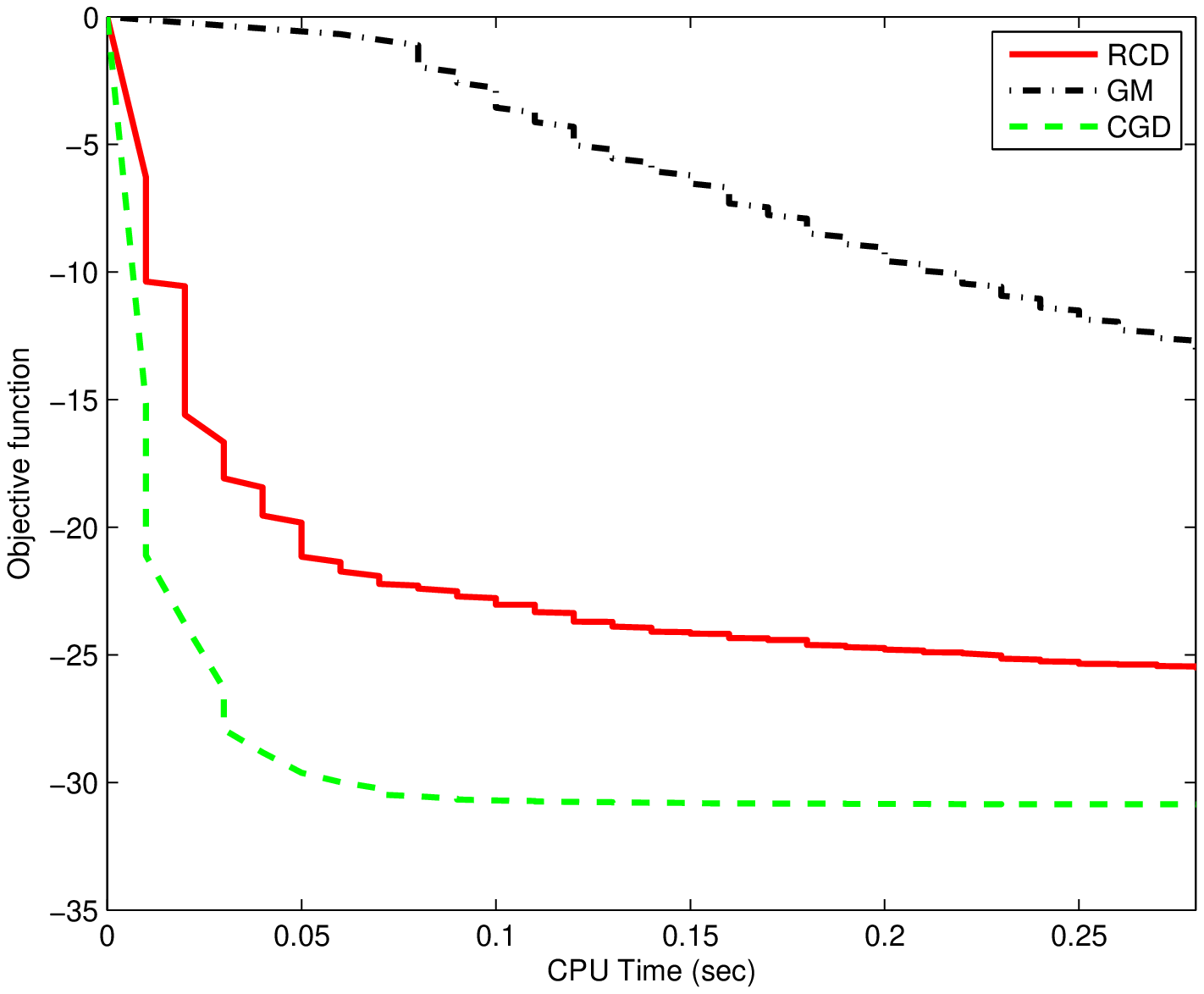}
\end{figure}

In Fig. 2 we present the performance of the three algorithms (RCD),
(GM) and (CGD)  on  a randomly generated matrix  $Z \in \rset^{2
\times 1000}$  for $50$ full-iterations with two different initial
points: $x^0 = e_1$ (the vector with the first entry $1$ and the
rest of the entries zeros) and $x^0 = \frac{e}{n}$. Note that for
the initial point $x^0 = e_1$, the algorithm (GM) is outperformed by
the other two methods: (RCD) and (CGD). Also, if all three
algorithms are initialized with $x^0 = \frac{e}{n}$, the algorithm
(CGD) has the worst performance among all three. We observe that our
algorithm (RCD) is very robust against the initial point choice.

In Fig. 3  we plot   the objective function evaluation over time (in
seconds) for the three algorithms  (RCD), (GM) and (CGD) on a matrix
$Z \in \rset^{30 \times 1000}$. We observe that the algorithm (RCD)
has a comparable performance with algorithm (GM) and a much better
performance than (CGD) when the initial point is taken
$\frac{e}{n}$. On the other hand, the algorithm (GM) has the worst
behavior among all three methods when sparse initializations are
used. However, the behavior of our algorithm (RCD) is not dependent
on   the sparsity of the initial point.

\begin{table}[h]
\centering \caption{Comparison of algorithms (RCD), (CGD) and (GM)
on Chebyshev center problems.} {\tiny
\begin{tabular}{| p{0.4cm} | p{0.8cm} | p{2.4cm} | p{2.4cm} | p{2.4cm} | p{0.3cm} |}
\hline
$x^0$ & \centering $n$ \hspace{15pt} $m$   & (RCD) & (CGD) &GM\\
\hline
&   &      full-iter/obj/time(sec)  &   iter/obj/time(sec)   &   iter/obj/time(sec) \\
\hline \hline
\multirow{6}{*}{$\frac{e}{n}$} & \centering 5 $\cdot$ $10^3$ 10 & 2064/-79.80/0.76
& 4620/-79.80/5.3 &  17156/-79.82/5.6 \\\cline{2-5}
& \centering $10^4$ \hspace{20pt} 10 & 6370/-84.71/4.75  &  9604/-84.7/23.2  &  42495/-84.71/28.01 \\ \cline{2-5}
& \centering 3 $\cdot$ $10^4$ 10  &  13213/-87.12/31.15  &  27287/-86.09/206.52  &  55499/-86.09/111.81 \\ \cline{2-5}
& \centering 5 $\cdot$ $10^3$ 30 & 4269/-205.94/2.75 & 823/-132.08/0.6 &  19610/-204.94/13.94\\ \cline{2-5}
& \centering $10^4$ \hspace{20pt} 30 &  5684/-211.95/7.51  & 9552/-211.94/33.42
&  28102/-210.94/40.18 \\ \cline{2-5}
& \centering 3 $\cdot$ $10^4$ 30 &  23744/-215.66/150.86  &   156929/-214.66/1729.1  &  126272/-214.66/937.33 \\
\hline \hline
\multirow{6}{*}{$e_1$} & \centering 5 $\cdot$ $10^3$ 10 &
 2392/-79.81/0.88  &  611/-80.8/0.77  &  29374/-79.8/9.6  \\ \cline{2-5}
& \centering $10^4$ \hspace{20pt} 10 &  9429/-84.71/7.05  &  350/-85.2/0.86  &  60777/-84.7/40.1\\
\cline{2-5}
& \centering 3 $\cdot$ $10^4$  10 & 13007/-87.1/30.64  &  615/-88.09/6.20   &  129221/-86.09/258.88\\\cline{2-5}
& \centering 5 $\cdot$ $10^3$ 30 & 2682/-205.94/1.73  &  806/-206.94/1.13 &  35777/-204.94/25.29 \\\cline{2-5}
 & \centering $10^4$ \hspace{20pt} 30  &  4382/-211.94/5.77  &  594/-212.94/2.14
& 59825/-210.94/85.52 \\ \cline{2-5}
& \centering 3 $\cdot$ $10^4$ 30  &  16601/-215.67/102.11  &  707 /-216.66/8.02  &   191303/-214.66/1421 \\
\hline
\end{tabular}
}
\end{table}
In Table 4, for a number of $n= 5\cdot 10^3, 10^4$ and $3\cdot 10^4$
points generated randomly using uniform distribution in $\rset^{10}$
and $\rset^{30}$, we compared  all three algorithms (RCD), (CGD) and
(GM) with two different initial points: $x^0=e_1$ and
$x^0=\frac{e}{n}$. Firstly, we  computed $f^*$ with the algorithm
(CGD) using $x^0 = e_1$  and imposed  the termination criterion
$f(x^k) - f(x^{k+1}) \le \epsilon$, where $\epsilon=10^{-5}$.
Secondly, we  used the precomputed optimal value $f^*$ to test the
other algorithms with termination criterion $f(x^k) - f^* \le 1$ or
$2$. We clearly see that our algorithm (RCD) has superior
performance over the (GM) method and is comparable with (CGD) method
when we start from $x^0= e_1$. When we start from $x^0 = \frac{e}{n}
$ our algorithm provides better performance in terms of objective
function and CPU time (in seconds) than the  (CGD) and (GM) methods
(at least $6$ times faster).  We also observe that our algorithm is
not sensitive w.r.t. the initial point.


\subsection{Random generated problems with $\ell_1$-regularization term}
\label{appl3} In this section we  compare  algorithm (RCD) with the
methods (CGD) and (GM) on problems with composite objective
function, where  nonsmooth part contains an $\ell_1$-regularization
term $ \lambda \sum_{i=1}^n |x_i|$. Many applications from signal
processing and data mining can be formulated into the following
optimization problem \cite{CanRom:06,QinSch:10}:
\begin{align}
  \min_{x \in \rset^n}  & \frac{1}{2}x^TZ^TZx + q^Tx + \left( \lambda \sum_{i=1}^n |x_i| + \mathbf{1}_{[l,u]}(x) \right) \\\label{probrand}
 & \text{s.t.:} \;\;\ a^T x=b, \nonumber
\end{align}
where $Z \in \rset^{m \times n}$  and  the penalty parameter
$\lambda > 0$. Further, the rest of the parameters are chosen as
follows:  $a=e$, $b=1$ and $-l=u =1$. The direction $d_{ij}$ at the
current point in the algorithm (RCD) is computed in closed form. For
computing the direction in the (CGD) and (GM) methods we need to
solve a double size quadratic knapsack problem of the form
\eqref{probiter}  that has linear time complexity \cite{Kiw:07}.

\begin{table}[h]
\centering \caption{Comparison of algorithms (RCD), (CGD) and (GM)
on  $\ell_1$-regularization problems.} {\tiny
\begin{tabular}{|p{0.4cm} |p{0.4cm} | p{0.7cm} | p{2.5cm} | p{2.5cm} | p{2.5cm} |}
\hline
$x^0$ & $\lambda$ & $n$ & (RCD) & (CGD) &(GM)\\
\hline
&   &    &  full-iter/obj/time(sec)  &   iter/obj/time(sec)   &   iter/obj/time(sec) \\
\hline \hline
\multirow{10}{*}{$\frac{e}{n}$} & \multirow{6}{*}{0.1}
& $10^4$ & 905/-6.66/0.87  &  10/-6.67/0.11   &  9044/-6.66/122.42  \\ \cline{3-6}
&& 5 $\cdot 10^4$ & 1561/-0.79/12.32  &  8/-0.80/0.686  & 4242/-0.75/373.99  \\ \cline{3-6}
&& $10^5$ &  513/-4.12/10.45  & 58/-4.22/7.55  & 253/-4.12/45.06  \\ \cline{3-6}
&& 5 $ \cdot 10^5$ & 245/-2.40/29.03  &  13/-2.45/9.20   &  785/-2.35/714.93  \\ \cline{3-6}
&& 2 $ \cdot 10^6$ & 101/-10.42/61.27  &  6/-10.43/22.79  &  1906/-9.43/6582.5  \\ \cline{3-6}
&& $10^7$ &  29/-2.32/108.58  &  7/-2.33/140.4   &  138/-2.21/2471.2  \\ \cline{2-6}
& \multirow{6}{*}{10}
& $10^4$ & 316/11.51/0.29  &  5858/11.51/35.67 & 22863/11.60/150.61  \\ \cline{3-6}
&& 5 $\cdot 10^4$ & 296/23.31/17.65  & 1261/23.31/256.6  & 1261/23.40/154.6 \\ \cline{3-6}
&& $10^5$ &  169/22.43/12.18  &  46/22.34/15.99  & 1467/22.43/423.4  \\ \cline{3-6}
&& 5 $ \cdot 10^5$ & 411/21.06/50.82 &  37/21.02/22.46   &  849/22.01/702.73  \\ \cline{3-6}
&& 2 $ \cdot 10^6$ & 592/11.84/334.30  &  74/11.55/182.44   &  664/12.04/2293.1  \\
\cline{3-6}
&& $10^7$ & 296/20.9/5270.2  &  76/20.42/1071.5   &  1646/20.91/29289.1    \\
\hline \hline
 \multirow{10}{*}{$e_1$} & \multirow{6}{*}{0.1} &
$10^4$
& 536/-6.66/0.51 & 4/-6.68/0.05   &  3408/-6.66/35.26  \\
\cline{3-6} && 5 $\cdot 10^4$ &  475/-0.79/24.30  &
84564/-0.70/7251.4 & 54325/-0.70/4970.7 \\ \cline{3-6}
&& $10^5$ & 1158/-4.07/21.43  &  118/-4.17/24.83 & 6699/-3.97/1718.2\\
\cline{3-6}
 && 5 $ \cdot 10^5$ & 226/-2.25/28.81  &  24/-2.35/29.03   & 2047/-2.25/2907.5 \\ \cline{3-6}
 && 2$ \cdot 10^6$ & 70/-10.42/40.4  &  166/-10.41/632   & 428/-10.33/1728.3  \\ \cline{3-6}
 && $10^7$ & 30/-2.32/100.1  &  *   &  376/-2.22/6731  \\ \cline{2-6} &
\multirow{6}{*}{10} & $10^4$ & 1110/11.51/1.03  &  17/11.52/0.14   &  184655/11.52/1416.8  \\ \cline{3-6}
&& 5 $\cdot 10^4$ & 237/23.39/1.22  &  21001/23.41/4263.5   &  44392/23.1/5421.4  \\
\cline{3-6} && $10^5$ & 29/22.33/2.47  &  *   &  *  \\ \cline{3-6}
&& 5 $ \cdot 10^5$ & 29/21.01/3.1  &  *   &  *  \\ \cline{3-6} &&2$
\cdot 10^6$ & 9/11.56/5.85  &  *   &  *  \\ \cline{3-6}
&& $10^7$ & 7/20.42/14.51  &  *   &  *    \\
\hline
\end{tabular}
}
\end{table}

In Table 5, for dimensions ranging from  $n= 10^4$  to $n = 10^7$
and for $m=10$, we  generated randomly the matrix $Z \in \rset^{m
\times n}$ and $q \in \rset^n$ using uniform distribution. We
compared all three algorithms (RCD), (CGD) and (GM) with two
different initial points $x^0=e_1$ and $x^0=\frac{e}{n}$ and two
different values of the penalty parameter $\lambda=0.1$ and $\lambda
= 10$. Firstly, we  computed $f^*$ with the algorithm (CGD) using
$x^0 = \frac{e}{n}$ and imposed  the termination criterion $f(x^k) -
f(x^{k+1}) \le \epsilon$, where $\epsilon=10^{-5}$. Secondly, we
used the precomputed optimal value $f^*$ to test the other
algorithms with termination criterion $f(x^k) - f^* \le 0.1$ or $1$.
For the penalty parameter $\lambda = 10$ and initial point $e_1$ the
algorithms (CGD) and (GM) have not returned any result within $5$
hours. It can be clearly seen from Table 5 that for most of the
tests with the initialization  $x^0 = e_1$ our algorithm (RCD)
performs up to $100$ times faster than the other two methods. Also,
note that when we start from $x^0 = \frac{e}{n}$ our algorithm
provides a comparable performance, in terms of objective function
and CPU time (in seconds), with algorithm (CGD). Finally, we observe
that algorithm (RCD) is the most robust w.r.t. the initial point
among all three tested methods.

\end{document}